\documentclass[english]{amsart}
\usepackage[utf8]{inputenc}
\usepackage[english]{babel}
\usepackage{amsthm,amsmath,amssymb,amsfonts,amscd,epsfig}
\usepackage{graphicx,enumerate,csquotes,epigraph}
\usepackage[all]{xy}
\usepackage{booktabs,delarray,mathtools,float,faktor,xfrac}
\usepackage{textgreek,verbatim,subfigure,tabularx,longtable,multirow,spalign}
\usepackage{footnote}\makesavenoteenv{tabular}\makesavenoteenv{table}
\usepackage{pgfplots}
\usepackage{tikz}
\usetikzlibrary{matrix}
\usetikzlibrary{cd,positioning,arrows,shapes}
\usepackage{mathrsfs}

\newtheorem{thm}{Theorem}[section]

\newtheorem{cor}[thm]{Corollary}
\newtheorem{lem}[thm]{Lemma}
\newtheorem{prop}[thm]{Proposition}
\theoremstyle{definition}
\newtheorem{defn}[thm]{Definition}
\theoremstyle{remark}
\newtheorem{rem}[thm]{Remark}
\numberwithin{equation}{section}
\newtheorem{examp}[thm]{Example}

\renewcommand{\epsilon}{\varepsilon}

\newcommand{\cL}{\mathcal L}

\newcommand{\cC}{\mathcal C}
\newcommand{\cF}{\mathcal F}

\newcommand{\cJ}{\mathcal J}
\newcommand{\cP}{\mathcal P}
\newcommand{\cA}{\mathcal A}
\newcommand{\bbR}{\mathbb R}
\newcommand{\bbT}{\mathbb T}
\newcommand{\bbC}{\mathbb C}
\newcommand{\bbN}{\mathbb N}
\newcommand{\bbZ}{\mathbb Z}

\newcommand{\bbL}{\mathbb L}
\newcommand{\mfg}{\mathfrak g}

\newcommand{\BS}{Bohr-Sommerfeld }

\newcommand{\bd}{\partial}

\newcommand{\frbd}[1]{\frac{\bd}{\bd #1}}
\newcommand{\frbdt}[2]{\frac{\bd #1}{\bd #2}}

\DeclareMathOperator{\rank}{rank}
\pgfplotsset{compat=1.14}

\tikzset{
  pics/torus/.style n args={3}{
    code = {
      \providecolor{pgffillcolor}{rgb}{1,1,1}
      \begin{scope}[
          yscale=cos(#3),
          outer torus/.style = {draw,line width/.expanded={\the\dimexpr2\pgflinewidth+#2*2},line join=round},
          inner torus/.style = {draw=pgffillcolor,line width={#2*2}}
        ]
        \draw[outer torus] circle(#1);\draw[inner torus] circle(#1);
        \draw[outer torus] (180:#1) arc (180:360:#1);\draw[inner torus,line cap=round] (180:#1) arc (180:360:#1);
      \end{scope}
    }
  }
}

\begin{document}

\title{Geometric quantization via cotangent models}
\author{Pau Mir}
\address{{Laboratory of Geometry and Dynamical Systems, Department of Mathematics}, Universitat Polit\`{e}cnica de Catalunya and BGSMath, Barcelona, Spain }
\email{pau.mir.garcia@upc.edu}
\author{Eva Miranda}
\thanks{{Pau Mir is supported by an FPI-UPC grant. Eva Miranda is supported by the Catalan Institution for Research and Advanced Studies via an ICREA Academia Prize 2016. Both authors are supported by the grants reference number PID2019-103849GB-I00 (AEI) and reference number 2017SGR932 (AGAUR).}}

\address{{Laboratory of Geometry and Dynamical Systems Department of Mathematics $\&$ Ins\-titut de Matemàtiques de la UPC-BarcelonaTech (IMTech )}, Universitat Polit\`{e}cnica de Catalunya $\&$ Centre de Recerca Matemàtica, Barcelona, Spain \\ and
\\ IMCCE, CNRS-UMR8028, Observatoire de Paris, PSL University, Sorbonne
Universit\'{e}, 77 Avenue Den\-fert-Rochereau,
75014 Paris, France}
\email{eva.miranda@upc.edu}

\begin{abstract}

In this article we give a universal model for geometric quantization associated to a real polarization given by an integrable system with non-degenerate singularities. This universal model goes one step further than the cotangent models in \cite{KiesenhoferMiranda} by both considering singular orbits and adding to the cotangent models a model for the prequantum line bundle. These singularities are generic in the sense that are given by Morse-type functions and include elliptic, hyperbolic and focus-focus singularities. Examples of systems admitting such singularities are toric, semitoric and almost toric manifolds, as well as physical systems such as the coupling of harmonic oscillators, the spherical pendulum or the reduction of the Euler’s equations of the rigid body on $T^*(SO(3))$ to a sphere. Our geometric quantization formulation coincides with the models given in \cite{mirandahamilton} and \cite{mirandapresassolha} away from the singularities and corrects former models for hyperbolic and focus-focus singularities cancelling out the infinite dimensional contributions obtained by former approaches. The geometric quantization models provided here match the classical physical methods for mechanical systems such as the spherical pendulum as presented in \cite{cushmansniatycki}. Our cotangent models obey a \emph{local-to-global principle} and can be glued to determine the geometric quantization of the global systems even if the global symplectic classification of the systems is not known in general.
\end{abstract}

\maketitle


\section{Introduction}

Quantization is a mathematical procedure which seeks to associate a quantum system to a classical Hamiltonian system by replacing functions by operators and Poisson brackets of functions by brackets of operators. Several paths have been traced for this passionate journey from geometry and analysis into Physics: geometric quantization, formal quantization, BRST quantization and semi-classical quantization, to cite a few. All of them supply Taylor-made master formulas to the day-dreamer mathematicians who are looking into the quantum world.
 
In this article we focus on the geometric quantization approach and we provide a new model which corrects former models and brings us closer to the role of quantization as a mathematical tamer of quantum physics. Some almost metaphysical questions still waft through the air: Can this be achieved? Do these methods depend on additional data? Can we find a universal model?

One of the virtues of our model is that it takes the cotangent bundle as a general set-up for our systems. The connection between a Hamiltonian system and the cotangent bundle is given by the cotangent lift and provides a unified approach to former attempts in the literature. On the other hand, one of the downfalls of our model is that, unlike other quantization models like Kähler quantization, it depends on choices (in our case, on the choice of a real polarization given by an integrable system) as it usually happens in the standard geometric quantization. One of the advantages of our method is that geometric quantization of integrable systems can be computed even if global classification of integrable systems with non-degenerate singularities is unknown in general (not even semi-local classification), as its recipe is based on gluing local models. So the \emph{"from local to global"} principle prevails here.

Geometric quantization and integrable systems are common mathematical objects on the interface of Geometry and Physics. Integrable systems represent a class of Hamiltonian systems which can be associated to an extra set of functions called first integrals, and are ubiquitous in Physics. Many known systems, such as any two dimensional system, or more complicated systems, such as the coupled harmonic oscillators or the spherical pendulum, are integrable. Other classical systems defined by attracting or repelling particles, such as Toda systems, are integrable. The geometric quantization procedure meets integrable systems when these are used as data attached to the geometric quantization process, in particular as providers of (real) polarizations.

In this article, we contemplate the quantization problem considering precisely the real polarization associated to an integrable system. The geometric quantization procedure starts with a prequantum complex line bundle $\bbL$, which is naturally associated to a symplectic manifold $(M^{2n},\omega)$ of integral class, and an attached connection $\nabla$ with curvature $\omega$. A \emph{flat section} $s$ of the line bundle is a solution to the equation $\nabla_X s=0$, where the derivation takes place along the direction $X$ of a polarization, which in this article is considered to be real and given by the integrable system. Flat sections form a sheaf, from which one constructs a cohomology that eventually gives the quantization.

Because of the maximum principle, an integrable system defined via smooth functions on a compact phase space must have singularities. Then, polarizations given by these systems are singular too. In this article, we analyze the contribution of singularities to the geometric quantization of singular integrable systems.

The simplest type of singularities of smooth functions are Morse-type singularities, which admit a Morse or a Morse-Bott normal form. For integrable systems on symplectic manifolds one might also demand a normal form in a neighbourhood of their singularities. That is to say, one might assume that there exist local coordinates such that the functions defining the system are simultaneously of Morse type and such that the symplectic form is Darboux. Those singularities were initially considered by Eliasson \cite{Eliasson90} and later by Miranda \cite{MirandaThesis} and Miranda-Zung \cite{MirandaZung}.

In former works by Hamilton \cite{Hamilton}, Hamilton-Miranda \cite{mirandahamilton} and Miranda-Presas-Solha \cite{mirandapresassolha}, the authors analyze the contributions of non-degenerate singularities of integrable systems to quantization. They find no contribution from elliptic points and infinite dimensional contributions for hyperbolic and focus-focus type singularities. Those infinite dimensional models clash with the initial expectations of obtaining a finite dimensional representation space as the quantization space (and thus, representation space) of a system defined on a compact manifold.

In this article, we work out "cotangent models" for integrable systems with non-degenerate singularities which can be of elliptic, hyperbolic and focus-focus type. This is a first step towards understanding the cotangent models of the pairs given by the polarization associated to such integrable systems and the prequantum line bundle. Those singularities naturally appear in polarizations on compact manifolds given by integrable systems of Morse-Bott type. In particular, any semitoric system (such as the ones studied in \cite{pelayo1,pelayo2}) gives rise to singularities of this type). These structures also show up naturally in algebraic geometry, for instance in the study of the K3 surface\footnote{A K3 surface is an example of a hyperkähler manifold with three compatible complex structures $i,j,k$. The denomination K3 comes from Kummer, Kähler and Kodaira and, according to André Weil, it is a reminiscence of the beautiful mountain K2 in Kashmir.}, which can be viewed as a semitoric system. When it comes to considering their quantization, several models have been proposed. However, none of them can compete with the model of Kähler quantization in terms of independence of the polarization and in terms of the principle \emph{quantization commutes with reduction, or simply $[Q,R]=0$} (see \cite{guilleminsternbergflag}, \cite{guilleminsternbergQR}, \cite{guilleminsternbergQR2}). Notwithstanding, Kähler quantization cannot always be applied since the conditions to have a polarization of Kähler type are not always fulfilled.

For regular integrable systems (without singularities) action-angle coordinates (the classical Arnold-Liouville-Mineur theorem) provide cotangent models as a neighbourhood of the Liouville torus can be symplectically interpreted as its cotangent bundle, $T^*(\mathbb T^n)$. This canonical identification gives a way to relate the choice of the Liouville $1$-form of the cotangent bundle with the connection $1$-form of the prequantum line bundle. In other words, the Liouville $1$-form of the cotangent bundle yields a canonical choice of the connection $1$-form. This connects the cotangent model to quantization in the regular case (see for instance \cite{MirandaPresas15} and \cite{sniatpaper}). With the ambition of extending these ideas to the singular set-up, we analyze the cotangent lift technique for different types of non-degenerate singularities (in the sense of Eliasson-Williamson) and provide brand-new cotangent models for the pair given by the polarization and the connection one-form.

Additionally, the existence of a local model of cotangent type allows to capture symmetry and is compatible with the $[Q,R]=0$ principle. The cotangent models used to define the new proposal for geometric quantization for non-degenerate singularities also allows to obtain a unique universal cotangent model. In contrast to the former models of geometric quantization for real polarizations endowed with non-degenerate singularities in \cite{mirandahamilton} and \cite{mirandapresassolha}, our new models provide finite dimensional representations for systems on compact manifolds which match the physical models.

One interesting advantage of our models is that they fit well with the sheaf-theoretical geometric quantization kit provided in \cite{MirandaPresas15}. In particular, the Künneth formula and Mayer-Vietoris recipe which were established there can be used to patch the cotangent models to provide global quantization on a compact manifold even if the global symplectic classification of non-degenerate integrable systems is still unknown in some cases. In other words, our cotangent models can be seen as building pieces of the geometric quantization puzzle as partitions of unity in Differential geometry allow us to invoke a local-to-global principle.

\textbf{Organization of this article:} In Section \ref{sec:background} we revise the rudiments of integrable systems, the main features of the theory of non-degenerate singularities of such systems, the cotangent lift technique and the basics of geometric quantization. In Section \ref{sec:cotangentmodels} as a novel result we present the local normal form theorem for integrable systems with non-degenerate singularities described as a product of lower dimensional (2 or 4) cotangent models. we introduce the notion of discrete cotangent lift and show that it generates the set of \BS leaves. In Section \ref{sec:cotangentmodelsforquantization} we connect cotangent models of Section \ref{sec:cotangentmodels} with the connection $1$-form of geometric quantization. We prove that an integrable system with only non-degenerate singularities of hyperbolic and elliptic type can be realized as a cotangent lift. In Section \ref{sec:newmodel} we propose a new local model, redefining the previous quantizations for the hyperbolic and focus-focus singularities. We unify the quantization procedure for non-degenerate singularities. In Section \ref{sec:conclusions} we discuss some applications to the quantization of K3 surfaces and the advantages of the new model. Finally, in Appendix \ref{sec:cohomcyl}, for the sake of completeness we compute explicitly the sheaf cohomology of the cylinder. We obtain the expression of the flat sections, the cochains and the cocycles by hand, a procedure that may help a non-familiar reader to understand the techniques applied in the proofs of Section \ref{sec:newmodel}.

\section{Background. Integrable systems, the cotangent lift, quantization}
\label{sec:background}

We provide in this section the basic concepts that will be necessary for the rest of the article, dividing them into three parts. First, we review some important results on integrable systems defined on symplectic manifolds and on classification of singularities. Then, we introduce the cotangent lift, a tool which extends group actions to the cotangent bundle. Finally, we give a complete scheme of geometric quantization.

\subsection{Moment maps, Hamiltonian systems and singularities}

Hamiltonian actions and the \emph{moment map} are the absolute key concepts in the link between symplectic geometry and integrable systems. In this section we give a brief review on them, giving special attention to integrable systems with non-degenerate singularities.

\begin{defn}
Let $H\in\cC^\infty(M^{2n})$ be a smooth function on a symplectic manifold $(M^{2n},\omega)$. The \textit{Hamiltonian vector field} $X_H$ associated to $H$ is defined as the only solution of $\iota_{X_H}\omega=-dH$.
\label{def:Hamiltonianvfsymplectic}
\end{defn}

\begin{rem}
In Physics, the Hamiltonian represents a function of the total energy of a system.
\end{rem}

\begin{defn}
An \textit{integrable system} on a symplectic manifold $(M^{2n},\omega)$ is given by a smooth map $f=(f_1,\dots,f_n):M^{2n}\rightarrow\bbR^n$ such that $\{f_i,f_j\}=\omega(X_{f_i},X_{f_j})=0$ for all $1\leq f_i,f_j\leq n$ and $\rank f=n$ almost everywhere.
\end{defn}

\begin{defn}
Let $G$ be a Lie group and $\mfg$ its Lie algebra. Consider also $\mfg^*$, the dual of $\mfg$. Suppose $\psi: G \to \mathrm{Diff}(M)$ is an action on a symplectic manifold $(M,\omega)$. It is called a \textit{Hamiltonian action} if there exists a map $\mu: M \to \mfg^*$ which satisfies:
\begin{itemize}
\item For each $X \in \mfg$, $d\mu^X = \iota_{X^{\#}}\omega$, i.e., $\mu^X$ is a Hamiltonian function for the vector field $X^{\#}$, where
\begin{itemize}
\item $\mu^X: p\longmapsto \langle \mu(p),X\rangle: M \longrightarrow \bbR$ is the component of $\mu$ along $X$,
\item $X^{\#}$ is the vector field on $M$ generated by the one-parameter subgroup $\{\exp tX \mid t \in \bbR\} \subset G$.
\end{itemize}
\item
The map $\mu$ is equivariant with respect to the given action $\psi$ on $M$ and the coadjoint action:
$\mu \circ \psi_g = \mathrm{Ad}_g^* \circ \mu$, for all $g \in G$.
\end{itemize}

Then, $(M,\omega,G,\mu)$ is called a \textit{Hamiltonian $G$-space} and $\mu$ is called the \textit{moment map}.
\end{defn}

The normal form of a completely integrable system around a whole leaf of a regular point is well-known by the Arnold-Liouville-Mineur theorem (see \cite{arnoldliouville} and \cite{mineur}).

\begin{thm}[Arnold-Liouville-Mineur, \cite{arnoldliouville}]
Let $(M^{2n},\omega)$ be a symplectic manifold. Let $\{f_1,\ldots,f_n\}$ be a set of $n$ functions on $M$ which are functionally independent ($df_1\wedge\dots\wedge df_n\neq 0$ on a dense set and pairwise in involution). Suppose that $m$ is a regular point of $F=(f_1,\ldots,f_n)$ and that the level set of $F$ through $m$, which we denote by $\mathcal F_m$, is compact and connected.

Then, $\mathcal F_m$ is a torus and on a neighbourhood $U$ of $\mathcal F_m$ there exist ${\bbR}$-valued smooth functions $(p_1,\dots,p_n)$ and ${\bbR}/{\bbZ}$-valued smooth functions $({\theta_1},\dots,{\theta_n})$ such that the following holds:
\begin{enumerate}
    \item The functions $(\theta_1,\dots,\theta_n,p_1,\dots,p_n )$ define a diffeomorphism $U\simeq\bbT^n\times B^{n}$.
    \item The symplectic structure can be written in terms of these coordinates as
    \begin{equation*}
       \omega=\sum_{i=1}^n d \theta_i \wedge dp_i.
    \end{equation*}
    \item The leaves of the surjective submersion $F=(f_1,\dots,f_{s})$ are given by the projection onto the
      second component $\bbT^n \times B^{n}$, in particular, the functions $f_1,\dots,f_s$ depend only on
      $p_1,\dots,p_n$.
\end{enumerate}
We call the $p_i$ \textit{action coordinates} and the $\theta_i$\textit{ angle coordinates}.
\label{thm:ALM}
\end{thm}

For singular points, where Arnold-Liouville-Mineur theorem does not apply, it is necessary to explore the moment map there at a local or semi-local level in order to obtain the topology of the whole singular leaf. It can be really difficult to understand both the geometry and the dynamics of the system depending on the degeneracy of $dF=(df_1,\ldots,df_n)$ but some results already do this work in the case of the simplest singularities, the non-degenerate singularities. For this type of singularities powerful classification results have been obtained and, for instance, we have local normal forms.

\begin{defn}
A point $m\in M^{2n}$ is \textit{singular} of an integrable Hamiltonian system given by $F=(f_1,\dots,f_n)$ if the rank of $dF=(df_1, \dots,df_n)$ at $m$ is less than $n$.
\label{def:ranksingpoint}
\end{defn}

\begin{defn}
Let $(M^{2n},\omega)$ be a symplectic manifold with an integrable Hamiltonian system of $n$ independent and commuting first integrals $f_1,\dots,f_n$. Consider a singular point $p\in M$ of rank $0$, i.e. $(df_i)_p=0$ for all $i$. It is called a \textit{non-degenerate singular point} if the operators $\omega^{-1}d^2f_1,\dots,\omega^{-1}d^2f_n$ form a Cartan subalgebra in the symplectic Lie algebra $\mathfrak{sp}(2n,\bbR)=\mathfrak{sp}(T_p M,\omega)$.
\label{def:non-degsingularpoint}
\end{defn}

The classification of non-degenerate singular points in the real case is equivalent to the classification of Cartan subalgebras and was obtained by Williamson \cite{Williamson}.

\begin{thm}[Williamson, \cite{Williamson}]
For a Cartan subalgebra $\cC$ of $\mathfrak{sp}(2n,\bbR)$, there exists a symplectic system of coordinates $(x_1,\dots, x_n, y_1,\dots,y_n)$ in $\bbR^{2n}$ and a basis of $n$ functions $f_1,\dots,f_n$ of $\cC$ such that each of the quadratic polynomials $f_i$ is one of the following:

\begin{align*}
&\quad f_i = x_i^2 + y_i^2 & {\rm for}& \ \ 1 \leq i \leq k_e \\
&\quad f_i = x_iy_i & {\rm for} &\ \ k_e+1 \leq i \leq k_e+k_h \\
&\begin{cases}f_i = x_i y_{i+1}- x_{i+1} y_i \\
f_{i+1} = x_i y_i + x_{i+1} y_{i+1}\end{cases}
& {\rm for} &\ \ i = k_e+k_h+ 2j-1, \ 1 \leq j \leq k_f
\label{eq:williamsonbasis}
\end{align*}

The three types are called \textit{elliptic}, \textit{hyperbolic} and \textit{focus-focus}, respectively.
\label{thm:WilliamsonCartanSubalgebra}
\end{thm}

The triple $(k_e,k_h,k_f)$ at a singular point of rank $k=n-k_e-k_h-2k_f$, called \textit{Williamson type} of the singularity, is an invariant of the point and an invariant of the orbit of the integrable system through the point \cite{Zung-AL1996}. The following result of Eliasson \cite{Eliasson90} and Miranda and Zung (\cite{MirandaThesis}, \cite{MirandaCentrEur}, \cite{MirandaZung}) extends the classification to Hamiltonian systems in symplectic manifolds.

\begin{thm}[Eliasson, Miranda, Zung, \cite{Eliasson90} \cite{MirandaThesis} \cite{MirandaCentrEur} \cite{MirandaZung}]
Let $F$ be an smooth integrable Hamiltonian system of degree $n$ on a symplectic manifold $(M^{2n},\omega)$. The Liouville foliation in a neighborhood of a non-degenerate singular point $m$ of rank $k$ and Williamson type $(k_e,k_h,k_f)$ is locally symplectomorphic to the foliation defined by the basis functions of Theorem~\ref{thm:WilliamsonCartanSubalgebra} plus regular functions $f_i=x_i$ for $i=k_e+k_h+2k_f+1$ to $n$.
\label{thm:locallinearization}
\end{thm}

This theorem can be extended to an orbit of the integrable system via the following two Theorems.

\begin{thm}[Model in a covering]
The symplectic manifold $(M,\omega)$ can be represented, locally at a non-degenerate singularity of rank $k$ and Williamson type $(k_e,k_h,k_f)$, as the direct product
$$\overbrace{M^{\text{reg}} \times\cdots \times M^{\text{reg}}}^{k}
\times \overbrace{M^{\text{ell}} \times \cdots \times M^{\text{ell}}}^{k_e}
\times \overbrace{M^{\text{hyp}} \times \cdots \times M^{\text{hyp}}}^{k_h}
\times \overbrace{M^{\text{foc}} \times \cdots \times M^{\text{foc}}}^{k_f}$$
Where:
\begin{itemize}
    \item $M^{\text{reg}}$ is a \emph{regular block}, representing the regular moment map given by $$f_r=x,$$
    \item $M^{\text{ell}}$ is an \emph{elliptic block}, representing the elliptic singularity given by $$f_e=x^2 + y^2,$$
    \item $M^{\text{hyp}}$ is an \emph{hyperbolic block}, representing the hyperbolic singularity given by $$f_h=xy,$$
    \item $M^{\text{foc}}$ is a \emph{focus-focus block}, representing the focus-focus singularity given by
    $$f_f=(f_1,f_2)= (x_1 y_{2}- x_{2} y_1,
    x_1 y_1 + x_{2} y_{2}).$$
\end{itemize}
For the first three types of blocks the symplectic form is $\omega = dx \wedge dy$, while for the focus-focus block it is $\omega = dx_1 \wedge dy_1 + dx_2 \wedge dy_2$.
\label{thm:directproduct}
\end{thm}

In the case of a smooth system (defined by a smooth moment map), a similar result was proved and described by Miranda and Zung in \cite{MirandaZung}. It summarizes some previously results proved independently and fixes the case where there are hyperbolic components ($k_h\neq 0$), because in this case the result is slightly different and it has to be taken the semi-direct product in the decomposition. Contrary to the case where there are only elliptic and focus-focus singularities, in which the base of the fibration of the neighbourhood is an open disk, if there are hyperbolic components the topology of the fiber can become complicated. The reason is essentially that for the smooth case a level set of the form $\{x_iy_i=\epsilon\}$ is not connected but consists of two components. Precisely, because of these two components, one cotangent model gives raise to two different local models whenever there is an hyperbolic singularity. To overcome this duplicity, one takes a quotient by a finite group $\Gamma$ (typically $\bbZ_2$). This is why an equivariant version in the presence of symmetries (Theorem \ref{th:MirandaZung}) yields the total classification.

\begin{thm}[Miranda-Zung, \cite{MirandaZung}]
Let $V = D^k \times \bbT^k \times D^{2(n-k)}$ with the following coordinates: $(p_1,...,p_k)$ for $D^k$, $(q_1 (mod 1),...,q_k (mod 1))$ for $\bbT^k$, and $(x_1,y_1,...,x_{n-k},y_{n-k})$ for $D^{2(n-k)}$ be a symplectic manifold with the standard symplectic form $\sum dp_i \wedge dq_i + \sum dx_j \wedge dy_j$. Let $F$ be the moment map corresponding to a singularity of rank $k$ with Williamson type $(k_e,k_h,k_f)$. There exists a finite group $\Gamma$, a linear system on the symplectic manifold $V/\Gamma$ and a smooth Lagrangian-fibration-preserving symplectomorphism $\phi$ from a neighborhood of $O$ into $V/\Gamma$, which sends $O$ to the torus $\{p_i=x_i=y_i = 0\}$. The smooth symplectomorphism $\phi$ can be chosen so that via $\phi$, the system-preserving action of a compact group $G$ near $O$ becomes a linear system-preserving action of $G$ on $V/\Gamma$. If the moment map $F$ is real analytic and the action of $G$ near $O$ is analytic, then the symplectomorphism $\phi$ can also be chosen to be real analytic. If the system depends smoothly (resp., analytically) on a local parameter (i.e. we have a local family of systems), then $\phi$ can also be chosen to depend smoothly (resp., analytically) on that parameter.
\label{th:MirandaZung}
\end{thm}

In summary, given an integrable system, there is a naturally associated Lagrangian foliation given by a distribution generated by the Hamiltonian. This result does not only classify integrable systems but also classifies Lagrangian foliations \cite{MirandaCentrEur}.

\subsection{The cotangent lift}
The cotangent bundle of a smooth manifold can be naturally equipped with a symplectic structure in the following way. Let $M$ be a differential manifold and consider its cotangent bundle $T^{*}M$. There is an intrinsic canonical linear form $\lambda$ on $T^{*}M$ defined pointwise by
\begin{equation*}
    \langle \lambda_p,v\rangle=\langle p,d\pi_pv\rangle, \hspace{25pt} p=(m,\xi)\in T^{*}M,v\in T_p(T^{*}M),
\end{equation*}
where $d\pi_p:T_p(T^{*}M)\longrightarrow T_mM$ is the differential of the canonical projection at $p$. In local coordinates $(q_i,p_i)$, the form is written as $\lambda=\sum_i p_i\,dq_i$ and is called the \textit{Liouville $1$-form}. Its differential $\omega=d\lambda=\sum_i dp _i\wedge dq_i$ is a symplectic form on $T^{*}M$.

\begin{defn}
Let $\rho:G\times M \longrightarrow M$ be a group action of a Lie group $G$ on a smooth manifold $M$. For each $g\in G$, there is an induced diffeomorphism $\rho_g:M\longrightarrow M$. The \textit{cotangent lift of $\rho_g$}, denoted by $\hat{\rho}_g$, is the diffeomorphism on $T^{*}M$ given by
\begin{equation*}
    \hat{\rho}_g(q,p):=(\rho_g(q),((d{\rho_g)}_q^{*})^{-1}(p)),\hspace{25pt}(q,p)\in T^*M
\end{equation*}
which makes the following diagram commute:
\begin{center}
\begin{tikzpicture}
  \matrix (m) [matrix of math nodes,row sep=4em,column sep=4em,minimum width=2em]
  {T^*M & T^*M \\
   M & M\\};
  \path[-stealth]
    (m-1-1) edge node [right] {$\pi$} (m-2-1)
            edge node [above] {$\hat{\rho}_g$} (m-1-2)
    (m-2-1) edge node [above] {${\rho}_g$} (m-2-2)
    (m-1-2) edge node [right] {$\pi$} (m-2-2);
\end{tikzpicture}
\end{center}
\end{defn}

Given a difeomorphism $\rho:M\longrightarrow M$, its cotangent lift preserves the canonical form $\lambda$ (see computations in \cite{MM}). Then, the canonical $1$-form is preserved by $\hat\rho$.

As a consequence:
$$\hat\rho^*(\omega)=\hat\rho^*(d\lambda)=d(\hat\rho^*\lambda)=d\lambda=\omega.$$
Meaning that the cotangent lift $\hat{\rho}_g$ preserves the Liouville form and the symplectic form of $T^*M$.

\begin{examp}
Let $\rho:(\bbR^3,+)\times \bbR^3 \to \bbR^3$ be the Lie group action corresponding to a space translation defined by $\rho_x(q)=q+x$. Write $(q,p)$ for an element of the cotangent bundle $T^*\bbR^3\cong\bbR^6$.

By definition, $\hat{\rho}_x$, the cotangent lift of $\rho_x$ is
\begin{equation}
\hat{\rho}_x(q,p)=(\rho_x(q),((d{\rho_x)}_q^{*})^{-1}(p))=(q+x,((Id^{*})^{-1}(p))=(q+x,p).
\label{examp:cotangentlifttranslation}
\end{equation}
    
\end{examp}

\begin{examp}
Let $\rho:SO(3,\bbR)\times \bbR^3 \to \bbR^3$ be a Lie group action defined by $\rho_A(q)=Aq$. Write $(q,p)$ for an element of $T_q^*\bbR^3$. By definition, $\hat{\rho}_A$, the cotangent lift of $\rho_A$ is $$\hat{\rho}_A(q,p)=(\rho_A(q),((d{\rho_A)}_q^{*})^{-1}(p))=(Aq,((A^{*})^{-1}(p))=(Aq,Ap),$$ where the last equality holds because $A$ is orthogonal. Like any cotangent lift, since the induced action in the cotangent bundle is Hamiltonian, it has an associated momentum map which, in this case, corresponds to the classical quantity $q\wedge p$.
\label{examp:cotangentliftrotation}
\end{examp}

\subsection{Overview on geometric quantization}

As a general principle, quantization consists in associating a Hilbert space $\mathcal{Q}$ to a symplectic manifold $(M,\omega)$. In geometric quantization, this Hilbert space is constructed using the sections of a complex line bundle $\bbL$. Normally, one declares as representation space flat sections of this bundle in some direction (given by a polarization). Such sections are not always defined globally along the leaves of the polarization. This is why it is convenient to use the sheaf-theoretic language (with sheaf meaning the sheaf of flat sections of the bundle) to surmount this difficulty. Kostant introduced the main ideas of geometric quantization in the 70s \cite{Kostant70} and, today, they remain useful and have applications in representation theory, a big variety of physical problems and many other fields. One of the main characters of the theory of geometric quantization are \BS leaves. Kostant's model goes through the cohomology associated to the sheaf of flat sections of $\bbL$ and is well-adapted for real polarizations given by integrable systems and toric manifolds, which are symplectic manifolds endowed with an effective Hamiltonian action of a torus whose rank is half of the dimension of the manifold \cite{MirandaCentrEur}.

An important result of Delzant, which connects quantization and moment maps, states the existence of a one-to-one correspondence between closed toric manifolds in dimension $2n$ and the \textit{Delzant polytope} on $\bbR^n$ \cite{delzant}. The Delzant polytope gives the real geometric quantization of closed toric manifolds \cite{Hamilton} because, given a toric manifold, its real geometric quantization is completely determined by the count of integer points in the interior of its associated Delzant polytope.

In the Kähler quantization, the Hilbert space of representations is defined from the set of holomorphic global sections of the complex line bundle and give $H^0(M,\bbL)$. In our case, we will use real polarizations, for which there are not globally defined leaf-wise flat sections. Therefore, we will need to look for the higher dimensional cohomology groups.

Segal also proposed a way to quantize a Hamiltonian system consisting essentially in associating to the phase space a real Hilbert space $(\cF,(,))$ of the states of one particle \cite{segal67}, bringing a symplectic structure $\omega$ and a complex structure $J$ such that the complexification $H$ of $\cF$ under $J$ has a complex scalar product $(\,,)_{\bbC}$ defined as $$(\,,)_{\bbC}= (\,,) +J\omega$$ and the Hamiltonian evolution of the system is expressed by a unitary flow.

Although Segal quantization is quite useful for many purposes, in this article we focus on Kostant's viewpoint, since it is a more convenient geometric quantization to deal with cotangent models. In particular, as we will provide cotangent models for regular leaves and non-degenerate singularities of integrable systems, it will be the right approach for this article.

\subsection{Prequantization and sheaf cohomology}

\begin{defn}
Let $(M,\omega)$ be a symplectic manifold. It is \textit{quantizable} if there exists a hermitian complex line bundle $\bbL$ over $M$ with a compatible connection $\nabla$ whose curvature is $\omega$. $\bbL$ is called the \textit{prequantization line bundle}.
\label{def:preqlinebundle}
\end{defn}

\begin{rem}
$(M,\omega)$ is quantizable if $\omega$ is exact, which occurs for cotangent bundles with the canonical symplectic structure and for compact manifolds if $[\omega]\in H^2(M,\bbZ)$.
\end{rem}

One additional piece of structure is required, a \textit{polarization}, to restrict which sections of $\bbL$ are considered, because the space of all sections is normally "too big".

\begin{defn}
A \textit{polarization} $\cP$ on a symplectic manifold $(M,\omega)$ is an integrable Lagrangian distribution in $TM\otimes \bbC$.
There are two polarizations in which we are mainly interested:
\begin{itemize}
    \item A \textit{real polarization}, in which $\Bar\cP=\cP$.
    \item A \textit{Kähler polarization} (or \textit{holomorphic polarization}), in which $\Bar\cP\cap\cP=\{0\}$ and the hermitian form
    $$i\omega(\cdot,\cdot):\Bar\cP\times\cP\longrightarrow C_{\bbC}^{\infty}(M)$$
    is positive definite.
\end{itemize}
\label{def:polarization}
\end{defn}

\begin{rem}
A \textit{real polarization} on a symplectic manifold $(M,\omega)$ is basically a foliation of $M$ into Lagrangian submanifolds.
\label{rem:realpolarization}
\end{rem}

We will usually want to compute the quantization of a compact completely integrable system $(M,\omega,F)$, using the singular real polarization given by the singular foliation by level sets of $F$, which are generically Lagrangian tori.

\begin{defn}
A section $\sigma$ of $\bbL$ is \textit{flat along the leaves} or \textit{leaf-wise flat} if it is covariant constant along the fibres of $F$, with respect to the prequantization connection $\nabla$. Namely, if $\nabla_X \sigma = 0$ for any vector field $X$ tangent to fibres of $F$. The set of sections which are flat along the leaves is a sheaf and it is denoted by $\cJ$.
\label{def:flat}
\end{defn}

We now recall the construction of the cohomology of sheaves or \textit{sheaf cohomology}, which is used to define the geometric quantization. We start defining presehaves and sheaves.

\begin{defn}
Let $X$ be a topological space. A \textit{presheaf} $\cF$ on $X$ assigns to every open set $U$ of $X$ an abelian group $\cF(U)$, usually called the set of \textit{sections} of $\cF$ over $U$.
It also assigns, to any $V \subset U$, a \textit{restriction map} $\cF(U) \to \cF(V)$, such that if $W \subset V \subset U$ and 
$\sigma \in \cF(U)$, then $$\sigma|_{W} = (\sigma|_{V})|_{W},$$
and if $V=U$ then the restriction is just the identity map.
\label{def:presheaf}
\end{defn}

\begin{defn}
A presheaf $\mathcal{J}$ is a \textit{sheaf} if the following properties hold:
\begin{enumerate}
\item For any pair of open sets $U$, $V$, and sections $\sigma \in \mathcal{J}(U)$
and $\tau \in \mathcal{J}(V)$ which agree on the intersection $U\cap V$, there exists a section $\rho \in \mathcal{J}(U\cup V)$ which restricts to 
$\sigma$ on $U$ and $\tau$ on $V$.

\item If $\sigma$ and $\tau$ in $\mathcal{J}(U\cup V)$
have equal restrictions to $U$ and $V$, then they are equal
on $U\cup V$.
\end{enumerate}
\label{def:sheaf}
\end{defn}

Now, we construct the cochains and the coboundary operator of the cohomology. Fix an open cover $\mathcal{A} = \{ A_\alpha\}$ of the manifold $M$. A \textit{\v Cech $k$-cochain} assigns, to each $(k+1)$-fold intersection of elements from the cover $\mathcal{A}$, a section of the sheaf $\mathcal{J}$, and we will denote this kind of intersection $A_{\alpha_0} \cap \cdots \cap A_{\alpha_k}$, where the $\alpha_j$ are distinct, simply by $A_{\alpha_0 \cdots \alpha_k}$. Then, a $k$-cochain is an assignment 
$f_{\alpha_0 \cdots \alpha_k} \in \mathcal{J}(A_{\alpha_0 \cdots \alpha_k})$
for each $(k+1)$-fold intersection in the cover $\mathcal{A}$. The set of $k$-cochains is denoted by $C^k_\mathcal{A}(M;\mathcal{J})$, or just by $C^k_\mathcal{A}$.

The coboundary operator $\delta$ that makes $C^\ast_\mathcal{A}$ into a cochain complex is defined in the following way:
\begin{equation}\label{defdelta}
 (\delta f)_{\alpha_0 \cdots \alpha_k} = 
\sum_{j=0}^{k} (-1)^j f_{\alpha_0 \cdots \hat{\alpha}_j \cdots \alpha_k}
|_{A_{\alpha_0 \cdots \cdots \alpha_k}},
\end{equation}
where the $\,\hat{}\,$ denotes that the index is omitted. Then, it is clear that if $f = \{f_{\alpha_0 \cdots \alpha_{k-1}} \}$ is a $(k-1)$-cochain, $\delta f$ is a $k$-cochain. With this definition, for instance, $(\delta f)_{123} = f_{23} - f_{13} + f_{12}$ and $(\delta \circ \delta f)_{123} = \delta (f_{23} - f_{13} + f_{12}) = f_{3} - f_2 - f_{3} + f_1 + f_{2} - f_1 = 0$. In general, $\delta\circ\delta=0$ and $C^\ast_\mathcal{A}$ is a well-defined cochain complex.

\begin{defn}
With the above definitions, the sheaf cohomology with respect to the cover $\mathcal{A}$ is the cohomology of this complex:
$$H^k_\mathcal{A} (M;\mathcal{J}) = \frac{\ker \delta^{k}}{\text{im}\, \delta^{k-1}},$$
where by $\delta^k$ denotes the map $\delta$ on $C^k_\mathcal{A}$.
\label{def:sheafcohomologywrtacover}
\end{defn}

The sheaf cohomology that is useful to work with is defined independently of the cover, the way of doing it is to take a limit over cover refinements. A cover $\mathcal{B}$ is a \textit{refinement} of a cover $\mathcal{A}$ if every element of $\mathcal{B}$ is a subset of some element of $\mathcal{A}$. A refinement provides a map $\rho \colon \mathcal{B} \to \mathcal{A}$, where $B \subset \rho(B)$ for all $B\in \mathcal{B}$, and gives a map $\phi \colon C^k_\mathcal{A}(U,\mathcal{J}) \to C^k_\mathcal{B}(U,\mathcal{J})$ induced by the restriction maps in the sheaf. Then, if $\eta \in C^k_\mathcal{A}$ is a cochain, $\phi \eta$ is defined by
$$(\phi \eta)_{B_0 B_1 \cdots B_k} = (\eta)_{(\rho B_0) (\rho B_1) \cdots ( \rho B_k)} |_{B_0 B_1 \cdots B_k}.$$

This map commutes with $\delta$ and induces a map on cohomology $H^*_\mathcal{A} \to H^*_\mathcal{B}$. All the possible choices of maps $\rho$ turn the collection of $H^\ast_\mathcal{A}$ for all open covers of $M$ into a directed system and the sheaf cohomology can be defined as the limit of this system, which can be proved to exist.

\begin{defn}
The \textit{sheaf cohomology} of $M$ is defined as the limit of the directed system:
$$H^*(M;\mathcal{J}) = \varinjlim H^*_\mathcal{A}(M;\mathcal{J}).$$
\label{def:sheafcohomology}
\end{defn}
There is a last result on sheaf cohomology that we will use, which is presented as Theorem 3.4 in \cite{MirandaPresas15}. It gives a classical Künneth formula also holds for the sheaf cohomology in a generalized form. Let $(M_1, \cJ_1)$ and $(M_2, \cJ_2)$ be a pair of symplectic manifolds endowed with Lagrangian foliations. The induced sheaf of flat sections associated to the product is denoted

There is a last result on sheaf cohomology that we will use, which is presented as Theorem 3.4 in \cite{MirandaPresas15}. It gives a classical Künneth formula also holds for the sheaf cohomology in a generalized form. Let $(M_1, \cJ_1)$ and $(M_2, \cJ_2)$ be a pair of symplectic manifolds endowed with Lagrangian foliations. The induced sheaf of flat sections associated to the product is denoted $\cJ_{12}$ and we call it \emph{product sheaf}. Then, there is a natural morphism of cohomology groups,
\begin{equation}
\Psi: H^*(M_1, \cJ_1) \otimes H^*(M_2, \cJ_2) \to H^*(M_1 \times M_2, \cJ_{12})
\end{equation}
induced by pull-back of the classes through the natural projections, and there is an isomorphism of cohomology groups 
$$ H^n(M_1 \times M_2, \cJ_{12}) \cong \bigoplus_{p+q=n} H^p(M_1, \cJ_1) \otimes H^q(M_2, \cJ_2), $$ if the geometric quantization associated to $(M_1,\cJ_1)$ has finite dimension, $M_1$ is compact and $M_2$ admits a good covering or is compact.

\subsection{Geometric quantization of regular and singular systems}

\begin{defn}
Let $(M,\omega,F)$ and $\bbL$ be as above and let $\cJ$ be the sheaf of flat sections. The \textit{quantization} of $M$ is
\begin{equation*}
\mathcal{Q}(M) = \bigoplus_{k\geq 0} H^k(M;\cJ).
\end{equation*}
\label{def:quant}
\end{defn}

\begin{defn}\label{d:bs}
A leaf $\ell$ of the (singular) foliation is a \textit{\BS leaf}
if there is a leaf-wise flat section $\sigma$ defined over all of $\ell$.
\end{defn}

Although leaf-wise flat sections always exist locally (because by construction the curvature of $\nabla$ is $\omega$, which is zero when restricted to a leaf), we are requiring global existence, which is a strong condition. The set of \BS leaves is discrete in the leaf space and a leaf is \BS if and only if its holonomy is trivial around all the loops contained in the leaf.

The main result about quantization using real polarizations is a theorem of \'Sniatycki~\cite{sniatpaper}

\begin{thm}[\'Sniatycki, \cite{sniatpaper}]
Let $(M^{2n},\omega)$ be a symplectic manifold with a prequantization line bundle $\bbL$. Take a real polarization $P$ such that the projection map $\pi \colon M \to B$ is a fibration with compact fibres. Then, $H^q(M;\cJ) = 0$ for all $q\neq n$. Furthermore, $H^n$ can be expressed in terms of the Bohr-Sommerfeld leaves and the dimension of $H^n$ is exactly the number of Bohr-Sommerfeld leaves.
\label{th:sniat}
\end{thm}

This result implies that, for a toric manifold foliated by fibres of the moment map, the \BS leaves correspond to the integer lattice points in the interior of moment polytope, excluding the ones on the boundary (see Figure \ref{fig:BStriangle}).

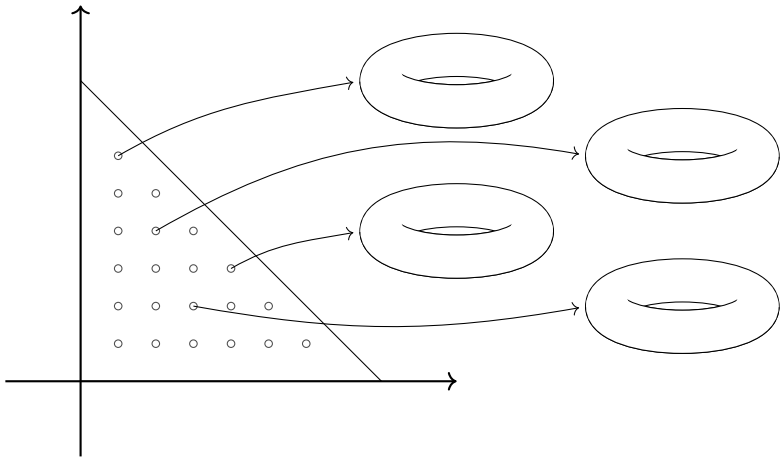
\begin{figure}[ht!]
\begin{center}\begin{tikzpicture}
\draw (0,0) -- (4,0) -- (0,4) -- (0,0);
\draw[thick, ->] (-1,0) -- (5,0);
\draw[thick, ->] (0,-1) -- (0,5);
\draw[black!70] (0.5,0.5) circle (0.05);
\draw[black!70] (0.5,1) circle (0.05);
\draw[black!70] (0.5,1.5) circle (0.05);
\draw[black!70] (0.5,2) circle (0.05);
\draw[black!70] (0.5,2.5) circle (0.05);
\draw[black!70] (0.5,3) circle (0.05);
\draw[black!70] (1,0.5) circle (0.05);
\draw[black!70] (1,1) circle (0.05);
\draw[black!70] (1,1.5) circle (0.05);
\draw[black!70] (1,2) circle (0.05);
\draw[black!70] (1,2.5) circle (0.05);
\draw[black!70] (1.5,0.5) circle (0.05);
\draw[black!70] (1.5,1) circle (0.05);
\draw[black!70] (1.5,1.5) circle (0.05);
\draw[black!70] (1.5,2) circle (0.05);
\draw[black!70] (2,0.5) circle (0.05);
\draw[black!70] (2,1) circle (0.05);
\draw[black!70] (2,1.5) circle (0.05);
\draw[black!70] (2.5,0.5) circle (0.05);
\draw[black!70] (2.5,1) circle (0.05);
\draw[black!70] (3,0.5) circle (0.05);
\node (T1) at (8,1) {};
\node (T1b) [left = of T1] {};
\node (T2) at (8,3) {};
\node (T2b) [left = of T2] {};
\node (T3) at (5,4) {};
\node (T3b) [left = of T3] {};
\node (T4) at (5,2) {};
\node (T4b) [left = of T4] {};
\pic at (T1) {torus={1cm}{2.8mm}{70}};
\pic at (T2) {torus={1cm}{2.8mm}{70}} {};
\pic at (T3) {torus={1cm}{2.8mm}{70}} {};
\pic at (T4) {torus={1cm}{2.8mm}{70}} {};
\draw [->] (0.5,3) to [out=30,in=190] (T3b);
\draw [->] (1,2) to [out=30,in=170] (T2b);
\draw [->] (2,1.5) to [out=30,in=190] (T4b);
\draw [->] (1.5,1) to [out=350,in=190] (T1b);
\end{tikzpicture}\end{center}
\caption{The integer points in the interior of Delzant's polytope correspond to the Bohr-Sommerfeld leaves, which are regular tori.}
\label{fig:BStriangle}
\end{figure}

In \cite{Hamilton}, Mark Hamilton calculated the sheaf cohomology of a toric manifold, a manifold equipped with a Lagrangian fibration with elliptic singularities. He obtained the explicit expression for the group that is non-zero in Theorem \ref{th:sniat}, which can be computed by counting over all non-singular Bohr-Sommerfeld fibres.

\begin{thm}[Hamilton, \cite{Hamilton}]
Let $(M,\omega)$ be a compact symplectic $2n$-manifold and suppose it is equipped with a locally toric singular Lagrangian fibration, with prequantization line bundle $\bbL$ and connection $\nabla$. Let $\cJ$ be the sheaf of leaf-wise flat sections of $\bbL$. Then, the cohomology groups $H^k(M;\cJ)$ are zero for all $k \neq n$, and $$H^n (M;\cJ) \cong \bigoplus_{b \in BS} \bbC$$ where the sum is taken over all non-singular \BS fibres.
\label{thm:hamlocaltor}
\end{thm}

In \cite{mirandahamilton}, Hamilton and Miranda proved the following key theorem for compact two-dimensional integrable systems with non-degenerate singularities.

\begin{thm}[Hamilton-Miranda, \cite{mirandahamilton}]
Let $(M, \omega, F)$ be a two-dimensional, compact, completely integrable system, whose moment map has only non-degenerate singularities. Suppose $M$ has a prequantum line bundle $\bbL$, and let $\cJ$ be the sheaf of sections of $\bbL$ which are flat along the leaves. The cohomology $H^1(M,\cJ)$ has two contributions of the form $\bbC^\bbN$ for each hyperbolic singularity, each one corresponding to a space of Taylor series in one complex variable. It also has one $\bbC$ term for each non-singular Bohr-Sommerfeld leaf. That is,
\begin{equation}\label{eq:mainthmn}
H^1 (M;\cJ) \cong \bigoplus_{p \in \mathcal{H}}
\bigl( \bbC^\bbN \oplus \bbC^\bbN\bigr)
\oplus \bigoplus_{b\in BS} \bbC_b.
\end{equation}
The cohomology in other degrees is zero. Thus, the quantization of $M$ is given by~\eqref{eq:mainthmn}.
\label{th:cohomology2dcompact}
\end{thm}

Two more important theorems were proved by Miranda, Presas and Solha in \cite{mirandapresassolha}, concerning the geometric quantization of focus-focus fibers.
\begin{thm}[Miranda-Presas-Solha, \cite{mirandapresassolha}]
The geometric quantization of a saturated neighborhood of a focus-focus fiber with $n$ nodes is:
\begin{itemize}
\item $0$ if the singular fiber is not Bohr--Sommerfeld.
\item isomorphic to
\begin{equation*}
\left(C^\infty(\bbR;\bbC)\right )^{n_f},
\end{equation*}if the singular fiber is Bohr--Sommerfeld, where $n_f=n$ (for compact fibers) and $n_f=n-1$ otherwise.
\end{itemize}
\label{th:ffquant}
\end{thm}

\begin{thm}[Miranda-Presas-Solha, \cite{mirandapresassolha}]
For a $4$-dimensional closed almost toric manifold $M$, with $n_r$ regular Bohr--Sommerfeld fibers and $n_f$ focus-focus Bohr--Sommerfeld fibers:
\begin{equation*}
\mathcal{Q}(M)\cong\bbC^{n_r}\oplus\left(\bigoplus_{j\in \{1,\dots,n_f\}} (C^\infty(\bbR ;\bbC))^{n(j)}\right),
\end{equation*}with $n(j)$ the number of nodes on the $j$-th focus-focus Bohr--Sommerfeld fiber.
\label{th:GQalmosttoric}
\end{thm}

And the expression of the flat sections on a focus-focus singular leaf is known (see \cite{solha}). These results allow to compute the quantization of some particular systems such as $K3$ surfaces \cite{mirandapresassolha}.

Although the geometric quantization of the regular case is quite clear and for specific systems with non-degenerate singularities of elliptic, hyperbolic and focus-focus type there do exist some results, a unified model is lacking and a specific proposal for it is what we introduce in this article.

\subsection{A useful example. The sheaf cohomology in the cylinder}
\label{sec:sheafcyl}
To illustrate the sheaf cohomology in a simple manifold, let us consider the cylinder $\bbR \times S^1$ and compute its sheaf cohomology. Set $M=\bbR\times S^1$, with polar coordinates $(t,\theta)$ and symplectic form $\omega = dt \wedge d\theta$.

We consider the real polarization of the symplectic manifold $M$ given by vectors tangent to the $S^1$ directions, i.e., the polarization $P$ is spanned by $\frbd{\theta}$ at each point of $M$.

Define the complex line bundle as $\bbL = M \times \bbC$. A section $\sigma$ of $\bbL$ can be seen as a complex-valued function on $M$. It is clear that $\omega = dt \wedge d\theta = d(t\,d\theta)$, so the connection defined by $$\nabla_X \sigma =X(\sigma) - \sigma i t\, d\theta(X),$$
with potential $1$-form $t\,d\theta$, will have curvature $\omega$ and makes $\bbL$ into a prequantization line bundle.

The leaf-wise sections have to satisfy $\sigma$ of $\bbL$ satisfy $$\nabla_X \sigma = X(\sigma) - \sigma \, i t \,d\theta(X) = 0$$ for any $X \in P$ or, equivalently,
$$\frbdt{\sigma}{\theta}-\sigma i t=0.$$

This partial differential is solved by $$\sigma(t,\theta)=a(t)e^{it\theta},$$
for any smooth function $a(t)$. These are the flatwise sections from which we can compute the sheaf cohomology $\cJ$ of $\bbL$.

We are looking for the set of \BS leaves, which are the leaves of the polarization $P$ that possess a non-trivial global leaf-wise flat section. The leaves of $P$ are of the form $\{t_0\}\times S^1$. Then, we want to see for which values of $t_0$ and $\theta$ the section $\sigma(t_0,\theta)=a(t_0)e^{it_0\theta}$, for $t_0$ fixed, is a well defined global section of the polarization.

Since $t_0$ is fixed along a leaf, $a(t_0)=a$ and $\sigma(t_0,\theta)=ae^{it_0\theta}$. The section $\sigma(t_0,\theta)$ has to be well defined in all the section, implying that $\sigma(t_0,\theta)=\sigma(t_0,\theta+2\pi)$. Therefore, $t_0$ has to satisfy $1=e^{2\pi it_0}$ and we get the condition $t_0\in\bbZ$. Then, the Bohr-Sommerfeld leaves are the leaves of $P$ of the form $\{k\}\times S^1$ with $m \in \bbZ$.

The space of global covariant constant sections over one leaf is one-dimensional, i.e., there is freedom in the choice of the value of $a\in\bbC$ in $\sigma=ae^{it_0\theta}$.

To compute the sheaf cohomology of $M$ we can determine the cohomology groups $H^k$ applying directly \'Sniatycki Theorem. In this case, we obtain that, for any open interval $I\subset\bbR$ and $U = I\times S^1 \subset M$,
$$
H^1(U,\cJ) \cong \bigoplus_{m \in \bbZ\cap I} \bbC, 
\qquad H^k(U,\cJ) = 0,\; k\neq 1,
$$
where $\cJ$ is the sheaf of flat sections $\sigma$.

But we can also compute the sheaf cohomology of the cylinder explicitly, as we do in Appendix \ref{sec:cohomcyl}.

\section{Integrable systems as cotangent lifts and the discrete cotangent lift}
\label{sec:cotangentmodels}

Cotangent lifts arise naturally in physics problems, and the link between integrable systems and cotangent models is clear in view of the following Kiesenhofer and Miranda result \cite{KiesenhoferMiranda}, which restates Theorem~\ref{thm:ALM} to reveal that at a semi-local level the regular leaves are equivalent to a completely toric cotangent lift model.

\begin{thm}[Kiesenhofer-Miranda, \cite{KiesenhoferMiranda}]
Let $F=(f_1,\dots,f_n)$ be an integrable system on a symplectic manifold $(M,\omega)$. Then, semi-locally around a regular Liouville torus, the system is equivalent to the cotangent model $(T^* \bbT^n)_{can}$ restricted to a neighbourhood of the zero section $(T^* \bbT^n)_0$ of $T^* \bbT^n$.
\label{thm:cotangentliftALM}
\end{thm}

In the classical models of the harmonic oscillator, the simple pendulum and the spherical pendulum one already finds the three different types of non-degenerate singularities in its lowest dimensional case. A simple elliptic singularity is appears in the harmonic oscillator, a simple hyperbolic singularity shows up in the simple pendulum and a simple focus-focus singularity arises in the spherical pendulum. The Hamiltonian vector fields associated each of the three non-degenerate singularities can be obtained from the cotangent lift of a Lie group action.

\begin{defn}
We denote by $\rho^h$ the following Lie group action:
$$\begin{array}{rccc}
  \rho^h:& \bbR\times\bbR &\longrightarrow &\bbR\\
  &\left(t,x\right) &\longmapsto &e^{-t} x
\end{array},$$
by $\rho^e$ the following Lie group action:
$$\begin{array}{rccc}
  \rho^e:& \bbR\times\bbC &\longrightarrow &\bbC\\
  &\left(t,z\right) &\longmapsto &e^{it} z
\end{array},$$
and by $\rho^f$ the following Lie group action:
$$\begin{array}{rcccl}
  \rho^f : &(S^1\times\mathbb{R})\times\mathbb{R}^2 &\longrightarrow & \mathbb{R}^2 &\\
  &\left((\theta,t),\begin{pmatrix}
            x_1\\
            x_2
            \end{pmatrix}\right)
            & \longmapsto &
            \begin{pmatrix}
            e^{-t} & 0\\
            0 & e^{-t}
            \end{pmatrix}
            \begin{pmatrix}
            \cos{\theta} & \sin{\theta}\\
            -\sin{\theta} & \cos{\theta}
            \end{pmatrix}
            \begin{pmatrix}
            x_1\\
            x_2
            \end{pmatrix}
\end{array}.$$
\label{def:cotliftactions}
\end{defn}

\begin{lem}
The infinitesimal generators of the cotangent lift of the Lie group actions $\rho^h,\rho^e,\rho^f$ coincide, respectively, with the vector fields corresponding to the hyperbolic, elliptic and focus-focus singularities.
\label{lem:cotangentmodels}
\end{lem}

\begin{proof}


For the hyperbolic singularity, take coordinates $(x,y)$ on $\bbR^2$ such that the symplectic form is $\omega=dx\wedge dy$ and the moment map is $f=xy$. The Hamiltonian vector field associated to $f$ is $X=(-x,y)$.

Consider the action of $\bbR$ on $\bbR$ given by:
$$\begin{array}{rccc}
  \rho^h:& \bbR\times\bbR &\longrightarrow &\bbR\\
  &\left(t,x\right) &\longmapsto &e^{-t} x
\end{array}.$$

Then, $((d\rho^h_t)_x^*)^{-1}$ acts as $y \longmapsto e^{t}y$. The cotangent lift $\hat{\rho}^h_t$ associated to the group action $\rho^h_t$, in coordinates $(x,y)$ of $T^{*}\bbR$, is:
$$\begin{array}{rccc}
\hat{\rho}^h: & T^{*}\bbR &\longrightarrow & T^{*}\bbR\\
    & \begin{pmatrix}x\\y\end{pmatrix}&\longmapsto &
    \begin{pmatrix}e^{-t} x\\e^t y \end{pmatrix}
\end{array}.$$

Deriving the last vector with respect to $t$ and evaluating at $t=0$, we obtain $X=(-x,y)$, the vector field associated to the hyperbolic singularity.


For the elliptic singularity, consider $\bbR^2$ with real coordinates $(x,y)$ and define the complex conjugate coordinates $(z,\bar{z})=(x+iy,x-iy)$ such that the symplectic form is $\omega=\frac{i}{2}dz\wedge d\bar{z}$. The moment map corresponding to the elliptic singularity is $f=\frac{1}{2}\left(x^2+y^2\right)=\frac{1}{2}z\bar{z}$. The Hamiltonian vector field associated to $f$ in the complex setting is $X=(iz,-i\bar{z})$.

Consider the following action of $\bbR$ on $\bbC$, which corresponds to a rotation of $z$ of angle $t$:

$$\begin{array}{rccc}
  \rho^e:& \bbR\times\bbC &\longrightarrow &\bbC\\
  &\left(t,z\right) &\longmapsto &e^{it} z
\end{array}.$$

Then, $((d\rho^e_t)_z^*)^{-1}$ acts as $\bar z \longmapsto e^{-it}\bar z$, and the cotangent lift $\hat{\rho}^e_t$ associated to the group action $\rho^e_t$, in coordinates $(z,\bar z)$ of $T^{*}\bbC$ is:
$$\begin{array}{rccc}
\hat{\rho}^e: & T^{*}\bbC &\longrightarrow & T^{*}\bbC\\
    & \begin{pmatrix}z\\ \bar z\end{pmatrix}&\longmapsto &
    \begin{pmatrix}e^{it} z\\e^{-it} \bar z \end{pmatrix}
\end{array}.$$

Deriving the last vector with respect to $t$ and evaluating at $t=0$ we obtain $X=(iz,-i\bar z)$, the vector field associated to the elliptic singularity.

The cotangent lift in the elliptic case uses a complex moment map which is not holomorphic. It is a formal development and holomorphicity is not assumed.


For the focus-focus singularity, take coordinates $(x_1,x_2,y_1,y_2)$ in $\bbR^4$ in such a way that the symplectic form is $\omega=dx_1\wedge dy_1+dx_2\wedge dy_2$ and the moment map is $F=(f_1,f_2)=(x_1y_2-x_2y_1,x_1y_1+x_2y_2)$.

The Hamiltonian vector fields associated to $f_1$ and $f_2$ are $X_1=(x_2,-x_1,y_2,-y_1),X_2=(-x_1,-x_2,y_1,y_2)$.

Let $G=S^1\times\mathbb{R}$ and $M=\mathbb{R}^2$. Consider the action of a rotation and a radial dilation of $\mathbb{R}^2$ given by
$$\begin{array}{rcccl}
  \rho^f : &(S^1\times\mathbb{R})\times\mathbb{R}^2 &\longrightarrow & \mathbb{R}^2 &\\
  &\left((\theta,t),\begin{pmatrix}
            x_1\\
            x_2
            \end{pmatrix}\right)
            & \longmapsto &
            \begin{pmatrix}
            e^{-t} & 0\\
            0 & e^{-t}
            \end{pmatrix}
            \begin{pmatrix}
            \cos{\theta} & \sin{\theta}\\
            -\sin{\theta} & \cos{\theta}
            \end{pmatrix}
            \begin{pmatrix}
            x_1\\
            x_2
            \end{pmatrix}
\end{array}.$$
Then, the cotangent lift $\hat{\rho}^f$ associated to the group action is:
$$\begin{array}{rccc}
\hat{\rho}^f :& T^{*}\mathbb{R}^2 &\longrightarrow & T^{*}\mathbb{R}^2\\
    &\begin{pmatrix}
            x_1\\
            x_2\\
            y_1\\
            y_2
            \end{pmatrix}           &\longmapsto&
            \begin{pmatrix}
            e^{-t}(x_1\cos{\theta}+x_2\sin{\theta})\\
            e^{-t}(-x_1\sin{\theta}+x_2\cos{\theta})\\
            e^t(y_1\cos{\theta}+y_2\sin{\theta})\\
            e^t(-y_1\sin{\theta}+y_2\cos{\theta})
            \end{pmatrix}
    \label{eq:diff3}
\end{array}.$$
Deriving the vector with respect to $\theta$ and evaluating at 0 we obtain $X_1=(x_2,-x_1,y_2,-y_1)$. While deriving the vector with respect to $t$ and evaluating at 0 we obtain $X_2=(-x_1,-x_2,y_1,y_2)$.
\end{proof}

In view of this realization of the three non-degenerate singularities as cotangent lifts of Lie group actions, we can automatically state a block form result, which is the cotangent analogous to Theorem~\ref{thm:directproduct}.

\begin{defn}
Let $p$ be a non-degenerate singularity of rank $k$ and Williamson type $(k_e,k_h,k_f)$ in a symplectic manifold $(M,\omega)$. Let $\rho^r,\rho^e,\rho^h,\rho^f$ be the Lie group actions of the regular case and the non-degenerate singular cases in Definition \ref{def:cotliftactions}. We define $\rho^p$ as the following composition of Lie group actions acting in the corresponding molecules:
$$\rho^p= (\rho^r:\bbR\times\bbT\to\bbR)^{k} \times (\rho^e:\bbR\times\bbC\to\bbC)^{k_e} \times (\rho^h:\bbR\times\bbR\to\bbR)^{k_h} \times (\rho^f:(S^1\times\bbR)\times\bbR^2\to\bbR^2)^{k_f}.$$

In detail, the map $\rho^p$ acts on the product manifold $(\bbR\times\bbT)^{k} \times (\bbR\times\bbC)^{k_e} \times (\bbR\times\bbR)^{k_h} \times ((S^1\times\bbR)\times\bbR^2)^{k_f}$ as the following composition:
$$\rho^p = \overbrace{\rho^r\circ\dots\circ\rho^r}^{k}\circ\overbrace{\rho^e\circ\dots\circ\rho^e}^{k_e}\circ\overbrace{\rho^h\circ\dots\circ\rho^h}^{k_h}\circ\overbrace{\rho^f\circ\dots\circ\rho^f}^{k_f},$$
with $\rho^r$ acting on $\bbR\times\bbT$, $\rho^e$ acting on $\bbR\times\bbC$, $\rho^h$ acting on $\bbR\times\bbR$ and $\rho^f$ acting on $(S^1\times\bbR)\times\bbR^2$ and all of them acting as the identity in the respective other components of the product of manifolds.
\label{def:blockcomposition}
\end{defn}

\begin{rem}
The cotangent lift $\hat\rho^p$ of $\rho^p$ is now naturally defined as the following action:
$$(\hat\rho^r:\bbR\times T^*\bbT\to T^*\bbT)^{k} \times (\hat \rho^e:\bbR\times T^*\bbC \to T^*\bbC)^{k_e} \times (\hat\rho^h:\bbR\times T^*\bbR\to T^*\bbR)^{k_h} \times (\hat\rho^f:(S^1\times\bbR)\times T^*\bbR^2 \to T^*\bbR^2)^{k_f}$$
\end{rem}

\begin{thm}[Cotangent model in a covering]
Take the definitions and notation of the section, let $(M,\omega)$ be a symplectic manifold. In a neighbourhood $U$ of a non-degenerate singularity $p$ of rank $k$ and Williamson type $(k_e,k_h,k_f)$, the manifold can be represented as the integral manifold of the infinitesimal generators of $\hat\rho^p$. Each component of the cotangent lift of $\rho^p$, when the Lie group element is fixed, acts on the corresponding blocks as presented in Theorem \ref{thm:directproduct}, i.e., $\hat\rho^r$ acts on ${M^\text{reg}}\cong T^*\bbT$, $\hat\rho^e$ acts on ${M^\text{ell}}\cong T^*\bbC$, $\hat\rho^h$ acts on ${M^\text{hyp}}\cong T^*\bbR$ and $\hat\rho^f$ acts on ${M^\text{foc}}\cong T^*\bbR^2$, and all of them acting as the identity in the respective other components of the product of manifolds.
\label{thm:directproductcotangent}
\end{thm}

\begin{proof}
The cotangent lift commutes with the product of manifolds and, hence, with the composition of functions in each block. The block decomposition form around a non-degenerate singularity of Theorem \ref{thm:directproduct} is then compatible with the cotangent models for each singularity in Lemma \ref{lem:cotangentmodels}, proving the cotangent lift block form.
\end{proof}

\begin{rem}
This theorem says that at each non-degenerate singularity, the symplectic manifold can be realized as a cotangent model thanks to the natural symplectic structure of the cotangent bundle and the existence of a cotangent lift model for each type of non-degenerate singularity. The associated moment map is:
$$F=(\overbrace{f_r,\dots,f_r}^{k},\overbrace{f_e,\dots,f_e}^{k_e},\overbrace{f_h,\dots,f_h}^{k_h},\overbrace{f_f,\dots,f_f}^{k_f}),$$
where $f_r, f_e, f_h, f_f$ are the elementary forms of the regular, elliptic, hyperbolic and focus-focus focus singularities respectively.
\end{rem}

\subsection{The discrete cotangent lift}

As a discrete analog of the the classical cotangent lift of a Lie group action, we define the \textit{discrete cotangent lift}, a tool which connects the cotangent models with the geometric quantization and allows to see \BS leaves as a cotangent lift in the classical sense.

On an integrable system of dimension $2n$, \BS orbits correspond to the integer points in the interior of the Delzant polytope. Since they are all of them $n$-dimensional tori, they can be seen as the orbit of a discrete translation action of a single torus.

Consider the integer points in the Delzant polytope of an integrable system, written in coordinates as $(x_1+m_1,\dots,x_n+m_n)\in\bbR^n$ for some fixed $(x_1,\dots,x_n)\in\bbR^n$ and a set $\mathcal{N}=\{(m_1,\dots,m_n)\in\bbZ^n\}$ of $n$-tuples. Here, $\bbR^n$ is the ambient space of the image of the moment map.
Consider the action
$$\begin{array}{rccc}
    \alpha: & \bbZ^n \times \bbR^n & \longrightarrow & \bbR^n\\
     & \big((m_1,\dots,m_n),(x_1,\dots,x_n)\big) & \longmapsto & (x_1+m_1,\dots,x_n+m_n)
\end{array},$$
which can be seen as the restriction to integer $n$-tuples in $\mathcal{N}$ of the associate continuous translation

$$\begin{array}{rccc}
    \tilde\alpha: & \bbR^n \times \bbR^n & \longrightarrow & \bbR^n\\
     & \big((t_1,\dots,t_n),(x_1,\dots,x_n)\big) & \longmapsto & (x_1+t_1,\dots,x_n+t_n)
\end{array}.$$

The orbit of $(x_1,\dots,x_n)\in\bbR^n$ by $\tilde\alpha$ restricted to $\mathcal{N}$ is the set of integer points in the interior of the Delzant polytope.

\begin{defn}
Consider the actions $\alpha$ and $\tilde\alpha$ as defined before. The \textit{discrete cotangent lift} $\hat\alpha: \bbZ^n \times T^*\bbR^n \to T^*\bbR^n$ of $\alpha$ is the restriction to $\bbZ^n \times T^*\bbR^n$ of the classical cotangent lift of the action $\tilde\alpha$.
\end{defn}

In practice, the discrete cotangent lift $\hat\alpha$ extends $\alpha$ to the cotangent bundle of $\bbR^n$.

\begin{lem}
The discrete cotangent lift coincides with the classical cotangent lift of $\tilde\alpha$ at the integer points of the basis.
\label{lem:discretecotliftcoincides}
\end{lem}

\begin{proof}
It is direct from the definition, since it is constructed as its restriction to the integer points.
\end{proof}

The discrete cotangent lifted action $\hat\alpha$ acts on the pairs of points in the interior of the Delzant polytope and tori in the manifold. Its orbit is not only the whole set of integer points in the Delzant polytope but the pairs of integer points and their associated \BS leaves (see Figure \ref{fig:BSaction}).

The action $\hat\alpha$ is well defined on the cotangent bundle of a vector space. Not only this, but it is also compatible with the symplectic structure of the cotangent bundle. To see this, consider the coordinates $(x_1,\dots,x_n)$ on the base and the symplectic dual coordinates $(y_1,\dots,y_n)$ on the fiber of $T^*\bbR^n$. Since the natural pairing of each coordinate $x_i$ with its dual $y_i$ coincides with the symplectic conjugation (reflected in the Liouville $1$-form $\lambda=\sum y_i dx_i $ and in the symplectic form $\omega = d\lambda = \sum dy_i\wedge dx_i$), the lift of the discrete action $\alpha$ preserves the symplectic structure.

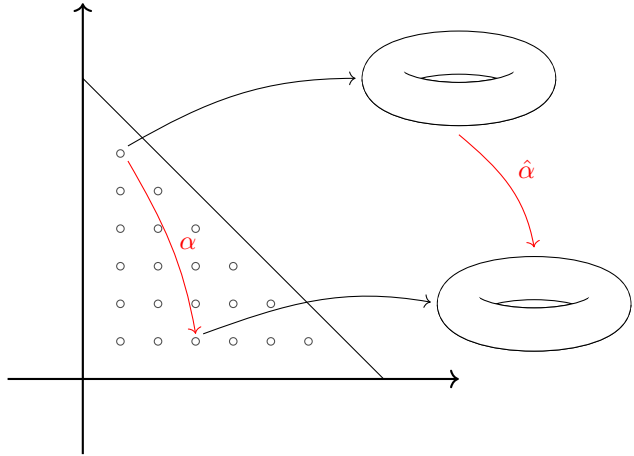
\begin{figure}[ht!]
\begin{center}\begin{tikzpicture}
\draw (0,0) -- (4,0) -- (0,4) -- (0,0);
\draw[thick, ->] (-1,0) -- (5,0);
\draw[thick, ->] (0,-1) -- (0,5);
\draw[black!70] (0.5,0.5) circle (0.05);
\draw[black!70] (0.5,1) circle (0.05);
\draw[black!70] (0.5,1.5) circle (0.05);
\draw[black!70] (0.5,2) circle (0.05);
\draw[black!70] (0.5,2.5) circle (0.05);
\draw[black!70] (0.5,3) circle (0.05);
\draw[black!70] (1,0.5) circle (0.05);
\draw[black!70] (1,1) circle (0.05);
\draw[black!70] (1,1.5) circle (0.05);
\draw[black!70] (1,2) circle (0.05);
\draw[black!70] (1,2.5) circle (0.05);
\draw[black!70] (1.5,0.5) circle (0.05);
\draw[black!70] (1.5,1) circle (0.05);
\draw[black!70] (1.5,1.5) circle (0.05);
\draw[black!70] (1.5,2) circle (0.05);
\draw[black!70] (2,0.5) circle (0.05);
\draw[black!70] (2,1) circle (0.05);
\draw[black!70] (2,1.5) circle (0.05);
\draw[black!70] (2.5,0.5) circle (0.05);
\draw[black!70] (2.5,1) circle (0.05);
\draw[black!70] (3,0.5) circle (0.05);
\node (T1) at (6,1) {};
\node (T1b) [left = of T1] {};
\node (T3) at (5,4) {};
\node (T3b) [left = of T3] {};
\pic at (T1) {torus={1cm}{2.8mm}{70}};
\pic at (T3) {torus={1cm}{2.8mm}{70}} {};
\draw [->] (0.6,3.1) to [out=30,in=180] (T3b);
\draw [->] (1.6,0.6) to [out=20,in=170] (T1b);
\draw [red,->] (0.6,2.9) to [out=300,in=100] (1.5,0.6);
\draw [red,->] (5,3.25) to [out=320,in=100] (6,1.75);
\node[red] (alpha) at (1.4,1.8) {$\alpha$};
\node[red] (hatalpha) at (5.9,2.8) {$\hat\alpha$};
\end{tikzpicture}\end{center}
\caption{The set of tori corresponding to the integer points in the interior of Delzant's polytope is the orbit of the action $\hat\alpha$.}
\label{fig:BSaction}
\end{figure}

Having introduced this language, we can reformulate Theorems \ref{th:sniat} and \ref{thm:hamlocaltor}, since in both \'Sniatycki's and Hamilton's cases there is a toric fibration.

\begin{thm}
Let $(F,M,\omega)$ be a toric integrable system in a compact symplectic manifold with a prequantization line bundle $\bbL$, equipped with a locally toric Lagrangian fibration. Let $\Delta$ be the Delzant polytope of the momentum map of $F$. Then, the set of \BS leaves coincides with the intersection of $\Delta$ with the image of the discrete cotangent lift of the action
$$\begin{array}{rccc}
    \alpha: & \bbZ^n \times \bbR^n & \longrightarrow & \bbR^n\\
     & \big((m_1,\dots,m_n),(x_1,\dots,x_n)\big) & \longmapsto & (x_1+m_1,\dots,x_n+m_n).
\end{array},$$
where $(m_1,\dots,m_n)$ is $n$-tuple of integers corresponding to any of the \BS leaves and $(x_1,\dots,x_n)$ can be transported via a symplectomorphism $\phi$ to the base coordinates of the Lagrangian fibration of $M$.
\label{th:BSdiscretecotlift}
\end{thm}

\begin{proof}
First, recall that in Theorem \ref{thm:directproductcotangent} we proved that the neighbourhood of every regular point and every non-degenerate singularity can be seen as a cotangent lift.

Then, since by Lemma \ref{lem:discretecotliftcoincides} the discrete cotangent lift can be thought as the restriction of the classical cotangent lift and its image is provided by the integer $n$-tuples.

Finally, Theorems \ref{th:sniat} and \ref{thm:hamlocaltor} prove that \BS fibers are isolated compact tori which are separated by integer values which form a lattice in the interior of the Delzant polytope. Then, in a toric manifold (with only elliptic singularities), the \BS fibers, which correspond to Lagrangian tori, are the image of the discrete cotangent lift of $\alpha$.
\end{proof}

The set of regular \BS leaves in the regular and toric case both coincides with the discrete cotangent lift of a lattice of the basis.

\section{Cotangent models for quantization}
\label{sec:cotangentmodelsforquantization}

In both the geometric quantization procedure and the cotangent lift technique there appears a $1$-form. In the first case it is the connection $1$-form $\Theta$ and in the second case it is the Liouville $1$-form $\lambda$, and in both cases their differential is the symplectic form $\omega$. It seems that there is a sort of freedom of choice for a $1$-form to satisfy that its differential is $\omega$, but the fact is that, in both cases, it is determined by precise conditions.

In the prequantization of the elliptic case (the cylinder of Section \ref{sec:sheafcyl}) one can already get an intuition of how the $1$-form is determined and which conditions it has to fulfill, apart from the obvious of having the symplectic form as its differential. In both the elliptic and the hyperbolic cases the conditions are essentially the same.

If we go back to the cotangent models in Section \ref{sec:cotangentmodels}, we see that it is natural to present the cotangent models associating them not only to a set of a symplectic manifold $(M^{2n},\omega)$ and an integrable system $F$ with non-degenerate singularities but also to a connection $\nabla$ with a connection $1$-form $\Theta$. We just proved that the connection $1$-form is a Liouville $1$-form of the type $\lambda = p\,dq$, where $p$ is the moment map "coordinate" (it could correspond to the singular one) and $q$ is the symplectic orthogonal to $p$.

\begin{thm}
Let $F$ be an integrable system with only non-degenerate singularities of elliptic and hyperbolic type defined in the prequantum line bundle of a symplectic manifold $(M^{2n},\omega)$. Let $p$ be a singular point of rank $k$ and Williamson type $(k_e,k_h,0)$. Then, in a neighbourhood of each point, there exists a unique product-like cotangent model such that the connection $1$-form $\Theta$ coincides with the Liouville $1$-form $\lambda=\sum p_i\,dq_i$, where the $p_i$ are the moment map coordinates and the $q_i$ are the symplectic orthogonal to $p_i$. The connection $1$-form $\Theta(q_1,\dots,q_n,p_1,\dots,p_n)$ of the prequantization of $F$ is determined up to a constant by the following conditions:
\begin{itemize}
    \item $d\Theta = \omega$,
    \item $\iota_{\frbd{p_i}}\Theta=\Theta\left(\frbd{p_i}\right)=0$, for $i=1,\dots,n$
    \item $\iota_{\frbd{q_i}}\Theta= \Theta\left(\frbd{q_i}\right)=F(p_i)$, for $i=1,\dots,n$.
\end{itemize}
\label{prop:cotangentmodelswithconnection}
\end{thm}

\begin{proof}
We prove that it is true for regular and simple elliptic and hyperbolic blocks on a $2$-manifold $(M^2,\omega)$. Since any singularity of Williamson type $(k_e,k_h,0)$ can be written as a product of these types of $2$-dimensional blocks, the $1$-form $\Theta$ will simply be the sum of the form at each block.

Around a regular point, the moment map writes simply as $f=x$ in Cartesian coordinates $(x,y)$. We take $\omega=dx\wedge dy$ and the connection $1$-form will be written as $\Theta(x,y)=\alpha(x,y)dx+\beta(x,y)dy$ for some smooth functions $\alpha$ and $\beta$. The second condition implies that $\alpha(x,y)=0$. Then, the third condition implies $\beta=\beta(x)$ and, finally, by the first condition we obtain that $\Theta_r=xdy$ is the $1$-form we were looking for.

The moment map can be written as $f_e=\frac{1}{2}(x^2+y^2)$ around an elliptic singularity in Cartesian coordinates; whilst in polar coordinates $(r,\theta)$ it can be written as $f_e=\frac{1}{2}r^2$. The symplectic form can be written as $\omega=dx\wedge dy=rdr\wedge d\theta$. Since we want to work with the coordinate of the moment map $f_e$, which is a regular function of the $r$ polar coordinate, we change to \emph{elliptic coordinates} $(s,\theta)$, which are essentially polar coordinates where with the transformation $s=\frac{1}{2}r^2$. Then, $\omega=ds\wedge d\theta$ and a connection $1$-form is written as $\Theta(s,\theta)=\alpha(s,\theta)ds+\beta(s,\theta)d\theta$ for some smooth functions $\alpha$ and $\beta$.

Since $\Theta$ has to satisfy $\Theta\left(\frbd{s}\right)=0$, $\alpha$ needs to be $0$. The condition $\Theta\left(\frbd{\theta}\right)=f_e(s)$ implies that $\beta=\beta(s)$. Finally, since we want that $d\Theta=\omega=ds\wedge d\theta$, we need to require that $\frbdt{\beta}{s}=1$. Then, $\Theta_{e}=sd\theta$ is the unique $1$-form satisfying the conditions.

The moment map around an hyperbolic singularity can be written as $f=xy$ in Cartesian coordinates, with symplectic form $\omega=dx\wedge dy$. In \emph{hyperbolic coordinates} $(h,b)$, where: $$h=xy,$$ $$b=-\frac{1}{2}\ln{\left\lvert\frac{x}{y}\right\rvert},$$ the connection $1$-form can be written as $\Theta=\alpha(h,b)dh+\beta(h,b)db$. Again, the three conditions together imply $\Theta_h= hdb$.

The definition of the Liouville $1$-form on a cotangent bundle of a smooth manifold is precisely built from the fibration position-momentum, and it finishes the proof. The $1$-form is, precisely:
$$\Theta=\sum_{i=1}^{k}\Theta_r+\sum_{i=k+1}^{k+k_e}\Theta_e+\sum_{i=k+k_e+1}^{k+k_e+k_h}\Theta_h,$$
with each $\Theta_j$ acting on the corresponding block in the decomposition of $M$ in $2$-blocks given by Theorem \ref{thm:directproduct}.
\end{proof}

\begin{rem}
In both cases we can interpret the contraction conditions as a symplectic orthogonality, which determines a unique choice of the $1$-form. In particular, the contraction of the $1$-form with the vector field parallel to the foliation of the moment map of the singularity gives precisely the moment map function, while its contraction with a vector field which is symplectic orthogonal to this foliation vanishes.
\end{rem}

With the notation of Section \ref{sec:cotangentmodels}, we can state this result as a cotangent model as follows.

\begin{thm}
Consider a triple $((M,\omega),F=(f_1,\dots,f_n),\Theta)$, where $(M,\omega)$ is a symplectic manifold, $F$ is an integrable system with only non-degenerate singularities of rank $k$ and Williamson type $(k_e,k_h,0)$ and $\Theta$ is a connection $1$-form. Then, locally at any singularity $p$, the system is equivalent to a triple where $\Theta$ is determined by the Liouville $1$-form of the cotangent bundle corresponding to the cotangent lift model of the singularity.
\label{th:triplemodel}
\end{thm}

\begin{cor}
The triple $((M,\omega),F=(f_1,\dots,f_n),\Theta)$ with the same properties as in Theorem \ref{th:triplemodel} locally decomposes as products of two dimensional triples. Explicitly, it decomposes as:
$$((M^\text{reg},\omega),f_r,\Theta_r)^{k} \times ((M^\text{ell},\omega),f_e,\Theta_e)^{k_e} \times ((M^\text{hyp},\omega),f_h,\Theta_h)^{k_h}.$$
\end{cor}

\begin{proof}
By Theorem \ref{thm:directproduct}, in a neighbourhood of a non-degenerate singularity, the manifold decomposes as the direct product of blocks which are all of them $2$-dimensional if there are no focus-focus type components. By Theorem \ref{thm:directproductcotangent}, the block decomposition model is a cotangent model, so the triple $((M,\omega),F=(f_1,\dots,f_n),\Theta)$ decomposes as products of two dimensional triples.
\end{proof}

\begin{rem}
By Theorem \ref{th:triplemodel}, the integrable system with only elliptic and hyperbolic singularities is automatically a cotangent lift.
\end{rem}

\section{Proposal of the new model of geometric quantization}
\label{sec:newmodel}

Jacques Vey proved in \cite{Vey78} that there was a unique complex model for the linearization of analytical systems. Indeed, he proved the following theorem in the holomorphic set-up.

\begin{thm}[Vey, \cite{Vey78}]
Let $(M^{2n},\omega)$ be an analytic complex symplectic manifold. Let $A$ be a Liouville algebra of critical function germs at a point $p\in M$. Then, there exist holomorphic coordinates $(p_1,\dots,p_n, q_1,\dots,q_n)$ in a neighbourhood of $p$ on $M$ such that $\omega=\sum_i dp_i\wedge dq_i$ and such that $A$ is the analytic algebra generated by the $n$ functions $h_i=p_iq_i$.
\end{thm}

From his idea of a unique model for a non-degenerate singularity in the complexes, we present a new model for geometric quantization. This model unifies the computation for the regular case and the three types of (real) non-degenerate singularities in the Williamson sense, which are equivalent in the complexes.

\subsection{Sheaf surgery for non-degenerate singularities}

Theorem \ref{thm:hamlocaltor} already unifies the geometric quantization of elliptic singularities with the regular case. We introduce some \emph{sheaf surgery} that will allow us to redefine the quantization of the hyperbolic and the focus-focus singularity to also bring together their geometric quantization.

To set up for the hyperbolic case, recall from \cite{mirandahamilton} the following result

\begin{prop}[Hamilton-Miranda, \cite{mirandahamilton}]
Let $Z$ be the neighbourhood of a hyperbolic singular point. If $\sigma \colon Z \to \bbL$ is a smooth leaf-wise flat section defined over $Z$, then $\sigma$ is Taylor flat at the singular point. That is,
\begin{equation*}
\frac{\bd^{j+k}\sigma}{\bd^j x\, \bd^k y}\biggr\rvert_{(0,0)} = 0
\quad \text{for all } j,k
\end{equation*}
\label{prop:flatat0}
\end{prop}

\begin{cor}
The only leaf-wise flat analytic section in a semi-local neighbourhood of a hyperbolic singularity is the zero section.
\label{cor:flatsectionshyperbolic}
\end{cor}

\begin{proof}
By Theorem \ref{prop:flatat0}, leaf-wise flat analytic sections in a neighbourhood of a singularity are the zero sections. Since the analytic extension of the zero section is zero, the only leaf-wise flat analytic section in a semi-local neighbourhood of a hyperbolic singularity is the zero section.
\end{proof}

Now, we construct a special sheaf by changing the sheaf of smooth flat sections in a neighbourhood on a hyperbolic singularity.

\begin{lem}
Let $p\in(M,\omega,F)$ be a non-degenerate singular point of hyperbolic type on an integrable system defined on a symplectic manifold. Suppose that it is equipped with a complex line bundle $\bbL$ and a connection $\nabla$ whose curvature is $\omega$. Then, the sheaf of leaf-wise smooth flat sections is still a sheaf when a in a neighbourhood of $p$ smooth sections are required to be analytic. We denote this sheaf by $\cJ_h$.
\label{lem:sheafhyperbolic}
\end{lem}

\begin{proof}
Analytic sections are a subclass of smooth sections. Observe that the process consists of two steps. Fisrt, thanks to the system linearization around a singularity of Eliasson \cite{Eliasson90}, the flat sections equation $\nabla_X\sigma=0$ is formed by analytic data (the field $X=x\frbd{x}-y\frbd{y}$ in dimension $2$) and the solutions sections are analytic. Requiring smooth sections to be analytic in a small neighbourhood does not change the fact that the intersection works well. This is because we are extending a piece where the sections are required to be zero and by Corollary \ref{cor:flatsectionshyperbolic}) this is compatible, since analyticity is a local condition. Therefore, conditions for being a sheaf (recall Definitions \ref{def:presheaf} and \ref{def:sheaf}) are satisfied.
\end{proof}

To set up for the focus-focus case, we need another kind of sheaf manipulation. In this case, we define a new sheaf by replacing a section in a neighbourhood.

\begin{lem}
Let $p\in(M,\omega,F)$ be a non-degenerate singular point of focus-focus type on an integrable system defined on a symplectic manifold. Suppose that it is equipped with a complex line bundle $\bbL$ and a connection $\nabla$ whose curvature is $\omega$. Then, the sheaf cohomology of leaf-wise flat smooth sections is still a sheaf cohomology when the cohomology of a neighbourhood of $p$ is replaced by the sheaf cohomology of a cylinder in that neighbourhood. We denote this sheaf by $\cJ_f$.
\label{lem:sheaffocus}
\end{lem}

\begin{proof}
Consider a semi-local neighbourhood $U$ of the singularity of focus-focus type $p\in(M,\omega,F)$. Apart from the singular focus-focus leaf $\cL$ containing $p$, the rest of the leaves of $U$ are regular tori. In the intersection of $\cL$ and a local neighbourhood $V$ of $p$ change the cohomology of the focus-focus leaf presented in \cite{mirandapresassolha} by the sheaf cohomology of the torus.

Explicitly, take $\omega=dt\wedge d\theta$ and take as flat sections the complex functions of the form $\sigma=a(t)e^{it\theta}$. For the complete construction of the cohomology, see Section \ref{sec:sheafcyl} and especially Appendix \ref{sec:cohomcyl}. The definition of these flat sections in the singular fiber automatically sets the cohomology of all the fibers in the entire $U$ into the torus cohomology. By continuity, this is a well-defined cohomology in $U$.

The cohomology glues back well because the focus-focus singularity is isolated and the topology of the complementary is glued back to a cylinder. So the pieces of the puzzle will glue back normally as in the computation of the sheaf cohomology in a neighbourhood of the torus. This is done using the Mayer-Vietoris formula in Section 4 of \cite{MirandaPresas15}. If this torus lies over an integer point of the lattice this would add a Bohr-Sommerfeld leaf to the computation.
\end{proof}

\begin{rem}
Observe that this argument works whenever there are several pinched nodes in the singular focus-focus fiber. In this case, the total count does not depend on the number of nodes of the multiple pinched torus. At each node of the pinched torus we redo the puzzle argument above to conclude.
\end{rem}

In practice, we do an interchange of a single cohomology piece, a focus-focus cohomology piece by a torus cohomology piece which is so close that the substitution can be done smoothly (see Figure \ref{fig:replacement}). This allows to define the geometric quantization of the focus-focus singularity as the quantization of the regular case.

We can think of the sheaf-theoretical approach as a puzzle construction where the manifold comes endowed with an adapted covering admitting local data (sections and polarization) and the invariants associated to the singularity, in this case the focus-focus singularity.

\begin{figure}[ht!]
    \centering
    \includegraphics[scale=0.12]{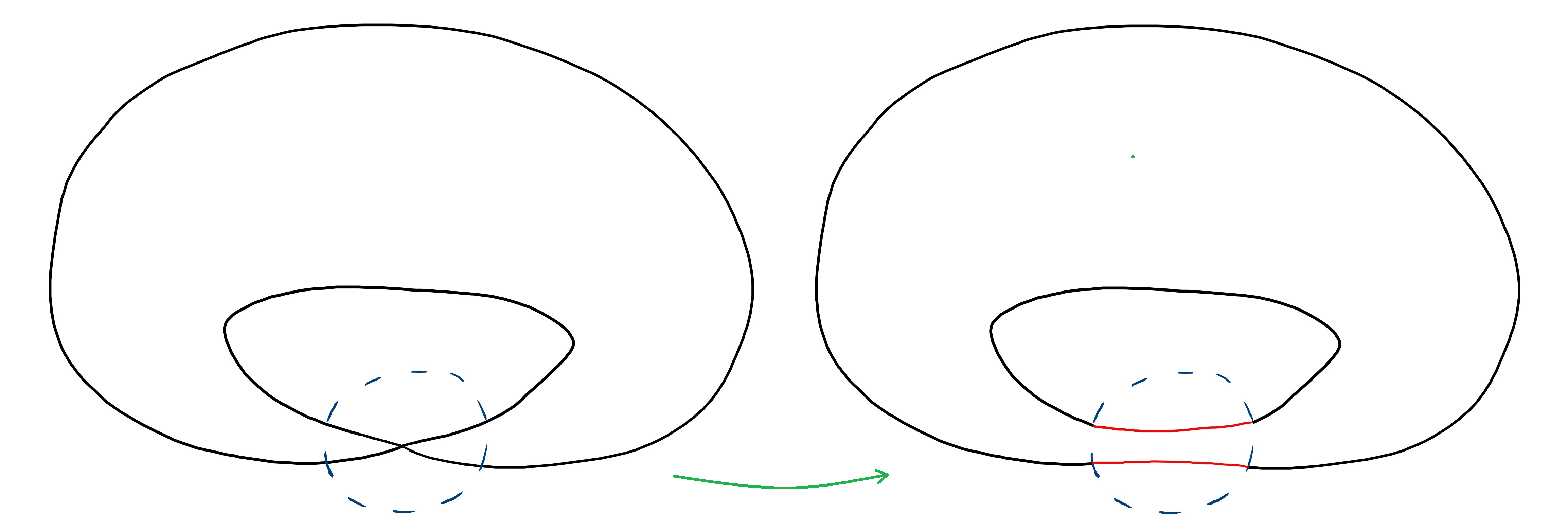}
    \caption{At practice, the model replaces the cohomology in a neighbourhood of the focus-focus singular point by the cohomology of the torus, like pieces of a puzzle.}
    \label{fig:replacement}
\end{figure}

The attachment of the new piece keeps the global cohomology of the neighbouring fibers well-defined. Understanding the sheaf cohomology computation in \cite{mirandapresassolha} under the computation kit provided in \cite{MirandaPresas15}, we can simply reconstruct the sheaf computation in the focus-focus case replacing the "puzzle piece" containing the focus-focus singularities by the torus piece. In practice, this will behave as replacing the puzzle by a red piece (regular cylinder) and gluing it back to the sheaf.

\subsection{Geometric quantization for non-degenerate singularities}

We propose a new model for geometric quantization which has which is convenient for non-degenerate singularities of both elliptic and hyperbolic type of certain characteristics. First, we introduce, the simple version for manifolds of dimension $2$ and later we generalize it to bigger dimensions. In dimension $2$, our model can be stated as follows.

\begin{thm}[Geometric quantization of non-degenerate singularities in dimension $2$]
Let $(M^2,\omega)$ be a symplectic $2$-manifold and suppose it is equipped with a non-degenerate singular Lagrangian fibration, with singularities of only elliptic and hyperbolic type, with a prequantization line bundle $\bbL$ and with connection $\nabla$. Let $\cJ$ be the sheaf of leaf-wise flat sections of $\bbL$. Then, the cohomology groups $H^k(M;\cJ)$ are zero for all $k \neq 1$, and $$H^1 (M;\cJ) \cong \bigoplus_{b \in BS} \bbC,$$ where the sum is taken over all non-singular \BS fibres.
\label{thm:gqnondegsingdim2}
\end{thm}

\begin{proof}
At the neighbourhood of the singular fibers, take the corresponding sheaf $\cJ$ of leaf-wise flat sections. Namely, take the sheaf of the cylinder and the corresponding cohomology (as in Appendix \ref{sec:cohomcyl}) if it is a singular elliptic fiber and take the sheaf $\cJ=\cJ_h$ of Lemma \ref{lem:sheafhyperbolic} and the corresponding cohomology if it is a singular hyperbolic fiber. In the regular components, just take the usual toric sheaf cohomology.

The quantization is computed as the sum of the cohomology groups, which are $0$ for $k\neq n$ and $H^n (M;\cJ) \cong \bigoplus_{b \in BS} \bbC$ where the sum is taken over all non-singular \BS fibres.
\end{proof}

For systems in higher-dimensional manifolds, we can still quantize both elliptic and hyperbolic singularities and even a merging of them. Nevertheless, we need that for each singularity of rank $k$, its Williamson type is $(k_e,k_h,0)$ with $k_h\leq 1$. We also have to require compactness of the manifold to apply a Künneth Formula for the product sheaf.

\begin{thm}[Geometric quantization of non-degenerate systems without focus-focus singularities]
Let $(M^{2n},\omega)$ be a compact symplectic $2n$-manifold and suppose it is equipped with a non-degenerate singular Lagrangian fibration, with singularities of rank $k$ and Williamson type $(k_e,k_h,0)$ with $k_h\leq 1$. Suppose it is also equipped with a prequantization line bundle $\bbL$ and with connection $\nabla$. Let $\cJ$ be the product sheaf of leaf-wise flat sections of $\bbL$. Then, the cohomology groups $H^k(M;\cJ)$ are zero for all $k \neq n$, and $$H^n (M;\cJ) \cong \bigoplus_{b \in BS} \bbC,$$ where the sum is taken over all non-singular \BS fibres.
\label{thm:gqnondegsing}
\end{thm}

\begin{proof}
Suppose that a singularity of rank $k$ has Williamson type $(k_e,0,0)$. Then, the result of quantization for toric systems (see Theorem \ref{thm:hamlocaltor}) applies and the quantization is given by $H^n (M;\cJ) \cong \bigoplus_{b \in BS} \bbC$.

Now, suppose that a singularity of rank $k$ has Williamson type $(k_e,1,0)$. At the neighbourhood of the singular fiber, consider the decomposition in elementary blocks given by Theorem \ref{thm:directproduct}, which in this case is $\overbrace{M^{\text{reg}} \times\cdots \times M^{\text{reg}}}^{k}
\times \overbrace{M^{\text{ell}} \times \cdots \times M^{\text{ell}}}^{k_e}
\times M^{\text{hyp}}$. Because of this decomposition, the sheaf $\cJ$ of leaf-wise flat sections of the prequantum line bundle is of product type. In other words, it is the product of the following components:
\begin{itemize}
    \item the elliptic components which are modelled as in Theorem \ref{thm:hamlocaltor} (and explicitly as in Appendix \ref{sec:cohomcyl});
    \item at the hyperbolic component where the sheaf $\cJ_h$ is given in Lemma \ref{lem:sheafhyperbolic} and the corresponding cohomology, recall that this sheaf is an extension of the zero section in a neighbourhood of the singularity;
    \item and in the regular components, just take the usual toric sheaf cohomology given by \ref{th:sniat}.
\end{itemize}

Because $M$ is compact and the sheaf can be seen as a product, we can apply the Künneth formula for geometric quantization (see Theorem 3.4 in \cite{MirandaPresas15}).

Therefore, in both cases the total cohomology is the sum of the cohomology of each block, and coincides with the quantization of the toric case because the hyperbolic component does not add any contribution. Then, the quantization is computed as the sum of the cohomology groups, which are $0$ for $k\neq n$ and $$H^n (M;\cJ) \cong \bigoplus_{b \in BS} \bbC$$ where the sum is taken over all non-singular \BS fibres.
\end{proof}

\begin{rem}
By the nature of our proof, this quantization works for $k_h\leq 1$. In cases where the multiplicity of the hyperbolic component of the singularity is bigger than $1$ the situation is not that clear because the sheaf can not always be realized as a product \cite{BolsinovIzosimov}. Nevertheless, we conjecture that the theorem holds in full generality because of the \textit{local to global principle}.
\end{rem}

In short, the theorem states that geometric quantization of an integrable system with only non-degenerate singular points without focus-focus singularities is locally given by the number of regular Bohr-Sommerfeld leaves in the neighbourhood of any singular point.

Applying this mantra to the set-up of \cite{MirandaPresas15} (where \'Sniatycki's Theorem is restated in a cotangent version, see Theorem 4.2 there) and \cite{mirandapresassolha} we obtain a simplification of the quantization formulas. Concerning the focus-focus case in dimension 4, at each focus-focus component we can apply Lemma \ref{lem:sheaffocus} and Theorem \ref{th:ffquant} now becomes:

\begin{thm}
The new geometric quantization of a saturated neighborhood of a focus-focus fiber with $n$ nodes is the same as the quantization of a regular fiber. It is $0$ if it is not \BS and contributes with the addition of $\bbC$ if it is \BS.
\label{th:ffquant2}
\end{thm}

This computation can be made global using again the Mayer-Vietoris formula presented in \cite{MirandaPresas15} (see Theorem 3.1. and Corollary 3.2 there) to obtain a new quantization model for almost toric manifolds. Theorem \ref{th:GQalmosttoric} becomes:

\begin{thm}
For a $4$-dimensional closed almost toric manifold $M$, with $n_r$ regular Bohr--Sommerfeld fibers and $n_f$ focus-focus Bohr--Sommerfeld fibers:
$$\mathcal{Q}(M)\cong\bbC^{n_r+n_f}.$$
\label{th:GQalmosttoric2}
\end{thm}

\section{Conclusions}
\label{sec:conclusions}

The new definition of the geometric quantization of a system with non-degenerate singularities of rank $k$ and Williamson type $(k_e,k_h,k_f)$ has advantages with respect to other definitions. From the physical point of view, the most important advantage is that it associates a discrete set to a model which is semi-locally compact. This is much more natural than to associate to it a continuous set. On the other hand, since the complexifications of the real, elliptic and hyperbolic singularities are equivalent, it is reasonable that the quantization of the focus-focus and hyperbolic cases essentially coincide with the quantization of the elliptic and the regular cases. The nodal trade, which can be thought as a smooth operator, allows the compatibility of this definition with the properties of the Delzant polytope.

Our definition, also, can be used to explicitly compute the quantization in particular examples, such as the $K3$ surface. In \cite{mirandapresassolha} the authors apply Kostant's model to singularities of focus-focus type to compute the cohomology groups associated to the real geometric quantization of a neighborhood of a focus-focus fiber of a 4-dimensional semitoric integrable system. They conclude that the first cohomology group is trivial but the second is not, since it is infinite dimensional if the singular fiber is Bohr--Sommerfeld (Theorem~\ref{th:ffquant}). As an application of Theorem~\ref{th:GQalmosttoric} in a particular example, they analyze the effect of nodal trades \cite{LeSy} in a K3 surface, which is an almost toric manifold.

In their example, a K3 surface is built from two copies of a symplectic and toric blow-up of the complex projective plane at 9 different points, applying nodal trades to all of their elliptic-elliptic singular fibers and taking their symplectic sum along the symplectic tori corresponding to the preimage of the boundary of their respective bases. This construction provides a K3 surface with 24 Bohr--Sommerfeld focus-focus fibers (see Figure \ref{fig:K3}). For this surface, they obtain that its quantization, with the previous model, is 
$\mathcal{Q}(K3)\cong\bbC^{26}\oplus\left( C^\infty(\bbR;\bbC)\right)^{24}$.

Then, the real geometric quantization of the K3 surface with the previous model from \cite{mirandapresassolha} is essentially different from the Kähler case, which is always finite dimensional. In the particular example we are considering, the dimension of its Kähler quantization is ${1}/{2}\cdot c_1(L)^2+2={1}/{2}\cdot(2\cdot 24+2\cdot 24)+2=50$, which is the same dimension given by our proposed model.

However, both models (Kähler quantization and the one provided in \cite{mirandapresassolha}) have something in common: The total number of Bohr-Sommerfeld leaves (26 regular +24 singular in \cite{mirandapresassolha}) coincides with the dimension 50 in the Kähler quantization.

In our model we are exchanging each of the infinity contributions from the focus-focus Bohr-Sommerfeld leaves by a finite dimensional contribution. Therefore, in our model, the representation space is
$\mathbb{C}^{50}$ and it coincides with the dimension provided by the Kähler quantization count.

Thus, a remarkable point of the K3 example is that our model for geometric quantization does coincide with the Kähler quantization, validating our new model and correcting the infinite dimensional contributions in \cite{mirandapresassolha}.

Other models for the quantization of K3 have been obtained from the perspective of Berezin-Toeplitz Quantization \cite{CastejonThesis}. In those cases, the representation space is the same as in the Kähler case, since one uses one of the compatible Kähler structures associated to the hyperkähler structure of K3. The Berezin-Toeplitz operators for the hyperkähler K3 in \cite{CastejonThesis} are built up from former work on the Kähler case in \cite{BT}.

\begin{figure}[ht]
\centering
\setlength{\unitlength}{1em}
\begin{tikzpicture}
\begin{scope}[scale=0.8]
\fill[fill=magenta!30] (0,0.3) circle (0.1);
\draw[black!30] (0,0.3) circle (0.15);
\fill[fill=magenta!30] (-2.15,-0.6) circle (0.1);
\draw[black!30] (-2.15,-0.6) circle (0.15);
\fill[fill=magenta!30] (-1.2,-0.2) circle (0.1);
\draw[black!30] (-1.2,-0.2) circle (0.15);
\fill[fill=magenta!30] (1.2,-0.2) circle (0.1);
\draw[black!30] (1.2,-0.2) circle (0.15);
\fill[fill=magenta!30] (2.15,-0.6) circle (0.1);
\draw[black!30] (2.15,-0.6) circle (0.15);
\fill[fill=magenta!30] (-2.5,-1.3) circle (0.1);
\draw[black!30] (-2.5,-1.3) circle (0.15);
\fill[fill=magenta!30] (2.5,-1.3) circle (0.1);
\draw[black!30] (2.5,-1.3) circle (0.15);
\fill[fill=teal!30] (0,1.5) circle (0.1);
\draw[black!30] (0,1.5) circle (0.15);
\fill[fill=teal!30] (-2.25,1) circle (0.1);
\draw[black!30] (-2.25,1) circle (0.15);
\fill[fill=teal!30] (-1.25,1.3) circle (0.1);
\draw[black!30] (-1.25,1.3) circle (0.15);
\fill[fill=teal!30] (1.25,1.3) circle (0.1);
\draw[black!30] (1.25,1.3) circle (0.15);
\fill[fill=teal!30] (2.25,1) circle (0.1);
\draw[black!30] (2.25,1) circle (0.15);
\fill[fill=black!30] (-1.05,-2.5) circle (0.1);
\fill[fill=black!30] (1.05,-2.5) circle (0.1);
\fill[fill=black!30] (0,-2.85) circle (0.1);
\fill[fill=black!30] (-1.05,2.5) circle (0.1);
\fill[fill=black!30] (1.05,2.5) circle (0.1);
\fill[fill=black] (0.5,1.5) circle (0.1);
\fill[fill=black] (-0.5,1.5) circle (0.1);
\fill[fill=black] (0,2.75) circle (0.1);
\fill[fill=black] (1.75,2) circle (0.1);
\fill[fill=black] (-1.75,2) circle (0.1);
\fill[fill=black] (0.5,-2) circle (0.1);
\fill[fill=black] (-0.5,-2) circle (0.1);
\fill[fill=black] (1.85,-2.2) circle (0.1);
\fill[fill=black] (-1.85,-2.2) circle (0.1);
\fill[fill=magenta] (0,-1.15) circle (0.1);
\draw (0,-1.15) circle (0.15);
\fill[fill=magenta] (-2.1,-1.2) circle (0.1);
\draw (-2.1,-1.2) circle (0.15);
\fill[fill=magenta] (-1.1,-1.15) circle (0.1);
\draw (-1.1,-1.15) circle (0.15);
\fill[fill=magenta] (1.1,-1.15) circle (0.1);
\draw (1.1,-1.15) circle (0.15);
\fill[fill=magenta] (2.1,-1.2) circle (0.1);
\draw (2.1,-1.2) circle (0.15);
\fill[fill=teal] (0,0) circle (0.1);
\draw (0,0) circle (0.15);
\fill[fill=teal] (-2.25,0.35) circle (0.1);
\draw (-2.25,0.35) circle (0.15);
\fill[fill=teal] (-1.25,0.15) circle (0.1);
\draw (-1.25,0.15) circle (0.15);
\fill[fill=teal] (1.25,0.15) circle (0.1);
\draw (1.25,0.15) circle (0.15);
\fill[fill=teal] (2.25,0.35) circle (0.1);
\draw (2.25,0.35) circle (0.15);
\draw (-2.9,0.6) arc (-120:90:0.15);
\fill[fill=teal] (-2.89,0.65) arc (-120:90:0.1);
\draw (2.9,0.85) arc (90:290:0.15);
\fill[fill=teal] (2.9,0.8) arc (90:290:0.1);
\draw (0,0) circle (3);
\draw (-3,0) arc (180:360:3 and 0.6);
\draw[black!30,dashed] (3,0) arc (0:180:3 and 1);
\fill[fill=black] (0,-0.6) circle (0.1);
\fill[fill=black] (-1.3,-0.55) circle (0.1);
\fill[fill=black] (-2.3,-0.4) circle (0.1);
\fill[fill=black] (1.3,-0.55) circle (0.1);
\fill[fill=black] (2.3,-0.4) circle (0.1);
\fill[fill=black] (-3,-0.2) arc (-90:90:0.1);
\fill[fill=black] (3,-0.005) arc (90:270:0.1);
\fill[fill=black!30] (0,1) circle (0.1);
\fill[fill=black!30] (-1.3,0.9) circle (0.1);
\fill[fill=black!30] (-2.3,0.65) circle (0.1);
\fill[fill=black!30] (1.3,0.9) circle (0.1);
\fill[fill=black!30] (2.3,0.65) circle (0.1);
\end{scope}
\end{tikzpicture}
\caption{K3 surface as a singular fiber bundle over the sphere. The preimage of the equator of the sphere contains 12 regular \BS fibers, obtained from the 12 elliptic singular \BS fibers which become regular after the symplectic sum.}
\label{fig:K3}
\end{figure}
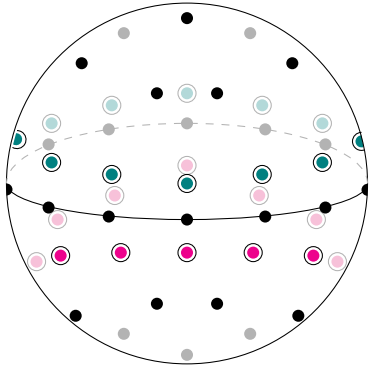

With this new model (see Figure \ref{fig:replacement}) we replaced the sheaf cohomology computation by the computation of the Bohr-Sommerfeld leaves, killing the infinite contributions of the focus-focus singularities and obtaining a quantization that coincides with the Kähler quantization of the nodal trade.

Our conclusions about glued blocks or molecules are coherent with the results of Bolsinov and Izosimov in \cite{BolsinovIzosimov}, where they build a particular integrable system with non-degenerate singularities. This system is homeomorphic to the direct product of a rank $0$ focus-focus singularity and a trivial fibration, but is not diffeomorphic to it. It is built as a family $\cF_t$ of double-pinched symplectic focus-focus singularities on $(M^4,\omega)$ that depends on a parameter $t \in (a,b) \subset \bbR$. This family generates a Lagrangian fibration on $M^6:= M^4 \times (a,b) \times S^1$, which is endowed with the symplectic structure $\omega + d t \wedge d \phi$. This fibration has a focus-focus singularity of rank $1$ with two critical circles on each fiber and, by construction, is homeomorphic to the direct product of a rank $0$ double-pinched focus-focus singularity and a regular foliation of an annulus by concentric circles.

Bolsinov and Izosimov show that there exists an invariant which should be a constant function for direct-product-type singularities and also constant for almost direct products. And, on the other hand, they prove that this invariant, in the case of the $M^6$ system, is a non-trivial function. Therefore, they prove that the singularity is not diffeomorphic to any almost direct product of the elementary non-degenerate singularities.

One interesting way of testing the applications of our quantization model is trying to answer the \emph{Can we hear the shape of a drum?} question raised by Mark Kac in \cite{MKac}. If the manifold is toric, we can effectively reconstruct it from the Bohr-Sommerfeld leaves, i.e., from the quantization as the Delzant polytope can be recovered from its integral interior points and the Delzant polytope completely determines the toric system \cite{delzant}. But if we do not know beforehand that the manifold does not have singularities of focus-focus type, we can not give the complete reconstruction because we would lack the information on whether the singular leaves are of toric type or of focus-focus type.

In our model for the focus-focus singularity, the singular focus-focus fiber does not add any contribution to the quantization when compared with respect to the elliptic model \cite{mirandapresassolha}. We wonder if it has a physical interpretation. Except for the singular one, in both models the fibers in a global neighbourhood of the singular one are topologically equivalent. This and the analysis of the focus-focus singularity which appears in the physical model of the spherical pendulum can explain it. At the top equilibrium of the pendulum, as it is pointed out by Cushman and \'Sniatycki in \cite{cushmansniatycki}, there is a singularity of focus-focus type. There, the conditions of the action functions of the integrable system force Planck’s constant to satisfy $2\pi\hbar = 8/n$, where $n$ is the first component of the lattice in the image of the moment map and takes some integer values. This is a really strong condition on Planck’s constant which is not likely to be satisfied in physics, since Planck's constant cannot equal a set of different values. Although for some computations Planck's constant is normalized in the sense that $2\pi\hbar = 1$ (see \cite{mirandahamilton} or \cite{Hamilton}), normalization would still not avoid this physical impossibility. In consequence, the Bohr-Sommerfeld conditions are not satisfied in the focus-focus point and it makes sense that it does not contribute to the quantization differently than a regular leaf.

\appendix
\section{The sheaf cohomology of the cylinder}
\label{sec:cohomcyl}

We reproduce here the calculations of the sheaf cohomology of the cylinder. In Section \ref{sec:newmodel} we change a piece of the cohomology and with this appendix we want to recall the whole picture of the construction of the covering and the cohomology groups of flat sections.

We start by calculating the cohomology on a finite subset of the cylinder $M=\bbR\times S^1$. Let $U$ be a \textit{band} of the cylinder, i.e., a subset of $M$ of the form $I\times S^1$, with $I \subset \bbR$ open and bounded. Assume that $U$ contains at most one Bohr-Sommerfeld leaf and construct a cover $\mathcal{A}$ of $U$ by three rectangles $A,B,C$ that slightly overlap as in Figure \ref{fig:cylindercover}.

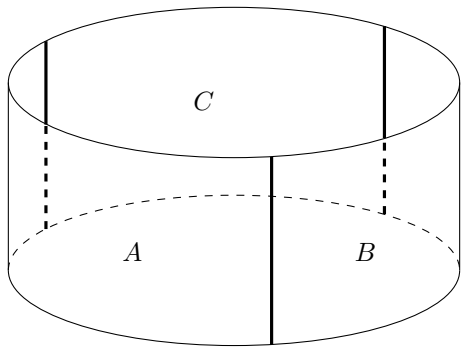
\begin{figure}[ht!]
\begin{center}\begin{tikzpicture}
\draw (0,0) ellipse (3 and 1);
\draw (-3,0) -- (-3,-2.5);
\draw (-3,-2.5) arc (180:360:3 and 1);
\draw [dashed] (-3,-2.5) arc (180:360:3 and -1);
\draw (3,-2.5) -- (3,0);
\draw [very thick] (0.5,-3.49) -- (0.5,-0.99);
\draw [dashed,very thick] (2,-1.75) -- (2,-0.74);
\draw [very thick] (2,-0.74) -- (2,0.75);
\draw [dashed,very thick] (-2.5,-1.95) -- (-2.5,-0.56);
\draw [very thick] (-2.5,-0.56) -- (-2.5,0.55);
\node at (-1.35,-2.25) {$A$};
\node at (1.75,-2.25) {$B$};
\node at (-0.4,-0.25) {$C$};
\end{tikzpicture}\end{center}
\caption{A band $U$ of the cylinder $M=\bbR\times S^1$ covered by three rectangles $A,B,C$. The rectangles overlap slightly in the thick joints.}
\label{fig:cylindercover}
\end{figure}

We want to identify the \v Cech $0$-cochains and $1$-cochains with respect to the cover $\cA$. A $0$-cochain $\alpha$ is an assignment of a flat section over $A$, $B$ and $C$ to that same subsets. Then, $\alpha$ assigns $A$ to the section $a_A(t)e^{it\theta}$, $B$ to $a_B(t)e^{it\theta}$ and $C$ to $a_C(t)e^{it\theta}$.

The angular coordinate $\theta$ can not be defined on all of $S^1$ so a branch of $\theta$ has to be fixed on each rectangle. We choose the branches so that $\theta_A = \theta_B$ on $A\cap B$, 
$\theta_B = \theta_C$ on $B\cap C$, and 
$\theta_C = \theta_A + 2\pi$ on $A\cap C$. The coboundary operator $\delta$ acts on $\alpha$ as
$$(\delta \alpha)_{ij} = \eta_j - \eta_i 
= a_j(t)\, e^{it\theta_j} - a_i(t)\, e^{it\theta_i},$$
for $i,j\in\{A,B,C\}$. We impose that, at the three intersections, $\delta \alpha$ is $0$, obtaining the following three equations:

\begin{align}
0 &= a_B(t)\, e^{it\theta_B} - a_A(t)\, e^{it\theta_A}
	\qquad\text{on } A\cap B\\
0 &= a_C(t)\, e^{it\theta_C} - a_B(t)\, e^{it\theta_B}
	\qquad\text{on } B\cap C\\
0 &= a_A(t)\, e^{it\theta_A} - a_C(t)\, e^{it\theta_C}
	\qquad\text{on } C\cap A
\end{align}

Then, $\alpha$ is a cocycle if and only if the following three equations simultaneously:
\begin{align}
a_B(t) &= a_A(t) \qquad\qquad\text{on } A\cap B\\
a_C(t) &= a_B(t) \qquad\qquad\text{on } B\cap C\\
a_A(t) &= a_C(t)\, e^{2\pi it} \qquad\text{on } C\cap A
\end{align}

which is not possible since $e^{2\pi it}$ can not equal $1$ in an entire interval of values of $t$. We conclude that are no $0$-cocycles and that $H^0 = 0$.

Now, a $1$-cochain $\beta$ is an assignment of a flat section over $A\cap B$, $B\cap C$ and $C\cap A$ to that same subsets. Then, $\beta$ assigns $A\cap B$ to the section $b_{AB}(t)e^{it\theta}$, $B\cap C$ to $b_{BC}(t)e^{it\theta}$ and $C\cap A$ to $b_{CA}(t)e^{it\theta}$. There only possible triple intersection in the cover $\cA$ is empty and $2$-cochains do not exist, implying that every $1$-cochain is a cocycle. Since the $1$-cochain is determined essentially by the three smooth functions $b_{ij}(t)$ on $I$, the space of $1$-cocycles is isomorphic to $C^\infty(I)^3$.

A $1$-cochain $\beta$ is a coboundary if there exists a $0$-cochain $\alpha = \{a_A(t) e^{it\theta_A}, a_B(t) e^{it\theta_B}, a_C(t) e^{it\theta_C}\}$ with $\delta \alpha = \beta$. Or, equivalently, if these three equations are satisfied:

\begin{align}
\beta_{AB} &= \alpha_B - \alpha_A \quad \text{on } A\cap B\\
\beta_{BC} &= \alpha_C - \alpha_B \quad \text{on } B\cap C\\
\beta_{CA} &= \alpha_A - \alpha_C \quad \text{on } C\cap A
\end{align}

Giving the sections explicitly, these equations transform to:

\begin{align}
b_{AB}(t)e^{it\theta_A} &= a_B(t) e^{it\theta_B} - a_A(t) e^{it\theta_A}
\qquad \text{on }A\cap B\\
b_{BC}(t)e^{it\theta_B} &= a_C(t) e^{it\theta_C} - a_B(t) e^{it\theta_B}
\qquad \text{on }B\cap C\\
b_{CA}(t)e^{it\theta_C} &= a_A(t) e^{it\theta_A} - 	a_C(t) e^{it\theta_C}
\qquad \text{on }C\cap A\label{eq:3rdeq}
\end{align}

Notice that, on each ordered intersection of two sets $E\cap F$, we use the $\theta$ coordinate from $E$. In each equation all the $\theta$ coordinates coincide except in Equation \ref{eq:3rdeq}, where they 
differ by a factor of $2\pi$. Then, we obtain a system of equations in the three unknown functions $a_A$, $a_B$, and $a_C$ which has to be true for each value of $t$ in $I$ and which has, as a matrix of coefficients, the following:
\begin{equation}
\begin{pmatrix}
-1&  1&  0\\
0&  -1&  1\\
e^{-2\pi it}&  0& -1
\end{pmatrix} 
\end{equation}
We observe that the matrix has rank $3$ (and therefore the system has a unique solution) when $e^{-2\pi it} \neq 1$. Then if $e^{-2\pi it}$ is never $1$ on $U$, every cocycle is a coboundary, and $U$ has the zero cohomology. Otherwise, if $e^{-2\pi it}=1$ somewhere in $I$,
which happens if and only if $I$ contains an integer $m$, the system only has a solution if the matrix \begin{equation}
\begin{pmatrix}
-1&  1&  0&  b_{AB}(m)\\
0&  -1&  1&  b_{BC}(m)\\
e^{-2\pi it}&  0& -1&  b_{CA}(m)
\end{pmatrix} 
\end{equation}
has rank $2$, i.e., if $\beta$ satisfies the condition
\begin{equation}
b_{AB}(m) + b_{BC}(m) + b_{CA}(m) = 0.
\end{equation}

Then, we have the following result:

\begin{prop}
the cohomology group $H^1$ of $U$ is precisely
\begin{equation}
H^1 = C^\infty (I)^3 / \{b_{AB}(m) + b_{BC}(m) + b_{CA}(m) = 0\},
\end{equation}
which is isomorphic to $\bbC$.
\end{prop} 

Observe that for $k>1$ there are no $(k+1)$-fold intersections in the cover $\cA$. Therefore, all the cohomology groups $H^k_{\cA}$ are zero for $k>1$.

The condition $e^{2\pi it}=1$ is satisfied exactly at the Bohr-Sommerfeld leaves, so we conclude that if $U$ is a band on the cylinder, the sheaf cohomology of $U$ with respect to the cover $\cA$ by the three rectangles is trivial if $U$ does not contain a \BS leaf, and it is: 
\begin{equation*}
H^k_{\cA} (U;\cJ) \cong 
\begin{cases}
  \bbC &k = 1\\
  0   &k\neq 1
\end{cases}
\end{equation*}

One can see that if another cover $\mathcal{B}$ of $U$ is made by $k$ rectangles instead of $3$, the cohomology calculated with respect to $\mathcal{B}$ is the same as that calculated with respect to $\cA$.

\bibliographystyle{alpha}
\bibliography{GeomQuantCotangentLift}

\end{document}